\documentclass[reqno]{amsart}
\usepackage{amssymb}
\usepackage{hyperref}
\usepackage{mathrsfs}

\usepackage[margin=1in]{geometry}  
\usepackage{graphicx}              
\usepackage{amsmath}               
\usepackage{amsfonts}              
\usepackage{amsthm}                
\usepackage{amssymb}

\usepackage{color}
\numberwithin{equation}{section}
\newtheorem{theorem}{Theorem}[section]
\newtheorem{lemma}[theorem]{Lemma} 
\newtheorem{proposition}[theorem]{Proposition}
\newtheorem{definition}[theorem]{Definition}
\newtheorem{remark}[theorem]{Remark}
\newtheorem{corollary}[theorem]{Corollary}

\allowdisplaybreaks

\numberwithin{equation}{section}

\usepackage{graphicx}              
\graphicspath{{./img/}}	

\begin{document}

\title[  Problems involving nonlocal integrodifferential operators without (AR) condition]{ Multiplicity results for elliptic problems involving nonlocal integrodifferential operators without Ambrosetti-Rabinowitz condition}

\author{Lauren Maria Mezzomo Bonaldo, Olímpio Hiroshi Miyagaki and Elard Juárez Hurtado}
\thanks{Department of Mathematics Universidade Federal de S\~ao Carlos.
Carlos,  S\~ao Carlos SP, Brazil.
E-mail: laurenmbonaldo@hotmail.com. Supported
by CAPES}\, \thanks{Department of Mathematics Universidade Federal de S\~ao Carlos.
Carlos,  S\~ao Carlos SP, Brazil.
E-mail: ohmiyagaki@gmail.com Supported
by CAPES} \,
\thanks{Department of Mathematics, Universidade de Brasília, Brasília,  DF, Brazil.
E-mail: elardjh2@gmail.com Supported
by CAPES}
\subjclass[2010]{35A15, 35B38, 35D3, 35J60, 35R11, 46E35, 47G2O}
\keywords{Nonlocal integrodifferential  operator, fractional Sobolev space with variable exponents,   Mountain Pass theorem, Fountain
theorem, Dual Fountain theorem}

\begin{abstract}
In this paper, we study  the existence and multiplicity of weak solutions for a general class of elliptic equations \eqref{e1.1} in a smooth bounded domain, driven by a nonlocal integrodifferential operator $\mathscr{L}_{\mathcal{A}K} $  with Dirichlet boundary conditions involving variable exponents without Ambrosetti and Rabinowitz type growth conditions.  Using different versions of the Mountain Pass Theorem, as well as, the Fountain Theorem and Dual Fountain Theorem with Cerami condition, we  obtain  the existence of weak solutions for the problem  \eqref{e1.1}  and we show that the problem treated has at least one nontrivial solution for any parameter $\lambda >0$ small enough as well as that the solution blows up, in the fractional Sobolev norm, as $\lambda \to 0$. 
Moreover, for the case sublinear, by imposing some additional hypotheses on the nonlinearity $f(x,\cdot)$, by using a new version of the symmetric Mountain Pass Theorem due to Kajikiya   \cite{kaji},  we obtain the existence of infinitely many weak solutions which tend to be zero, in the  fractional 
Sobolev norm, for any parameter $\lambda >0$.
As far as we know, the results of this paper are new in the literature.
\end{abstract}
\date{}
\maketitle
\section{Introduction and main results}\label{int}
 \hfill \break
In recent years,  the study of elliptic problems involving fractional operators with variable exponents has become the object of study of many researchers. Problems of this kind have a very interesting theoretic basis very interesting whose integrability and analytical structure require delicate proofing techniques as we can see in \cite{anaour,bahrouni,ky,kaufmann}. As well as,  they have concrete applications in the most diverse fields as optimization, finance, phase transitions, continuum mechanics, image process, game theory, crystal dislocation, quasi-geostrophic flows, peridynamic theory,  see \cite{aut,trud,mol,bucu,caff,ha,elard,ian,palatucci,pati,vaal} and their references. 

 In this sense, in order to expand the results around this theory, in the present paper we are concerned with the existence and multiplicity of weak solutions for elliptic equations involving the nonlocal integrodifferential operators with variable exponents. Namely, we consider the following problem
 \begin{equation}\tag{$\mathscr{P}_{\lambda}$}\label{e1.1}
\left\{\begin{array}{rcll}
\mathscr{L}_{\mathcal{A}K} u&=&\lambda f(x,u) & \textrm{ in } \Omega,\\
u&=&0 & \textrm{ in }\mathbb{R}^{N}\setminus\Omega,
\end{array}\right.
\end{equation}
 where $\lambda>0$ is a real parameter, $\Omega\subset\mathbb{R}^{N},$ $N\geqslant2,$ is a smooth bounded domain, and to define the nonlocal integrodifferential operator $\mathscr{L}_{\mathcal{A}K}$ we will need the variable exponents  $p(x):=p(x,x)$ for all $x \in \mathbb{R}^{N}$ with $p \in C(\mathbb{R}^{N} \times\mathbb{R}^{N})$  satisfying: 
  \begin{equation}\label{a23}\tag{$\mathit{p_1}$}
  \begin{split}
   & p \mbox{ is symmetric, that is, } p(x,y)=p(y,x),  \\&  1< p^{-} :=\displaystyle{\inf _{(x,y) \in \mathbb{R}^{N}\times \mathbb{R}^{N}}}\,p(x,y) \leqslant \displaystyle{\sup _{(x,y)\in \mathbb{R}^{N}\times \mathbb{R}^{N}}}\,p(x,y):= p^{+}<\frac{N}{s},\hspace{0.2cm} s \in(0,1),
 \end{split}
  \end{equation}
  and we consider  the fractional critical variable exponent related to $p\in C(\mathbb{R}^{N} \times\mathbb{R}^{N})$ defined by $\displaystyle{p^{\star}_{s}(x)= \frac{Np(x)}{N-sp(x)}}.$ 
  
 The nonlocal integrodifferential operator $\mathscr{L}_{\mathcal{A}K}$ is defined on suitable fractional Sobolev spaces (see Section 2) by
 \begin{equation*}
  \mathscr{L}_{\mathcal{A}K}u(x)= P.V. \int_{\mathbb{R}^{N}}\mathcal{A}(u(x)-u(y))K(x,y)\,dy, \hspace{0.2cm} x \in \mathbb{R}^{N},
  \end{equation*}
  where $P.V.$ is the principal value.
   
 The map $\mathcal{A}:\mathbb{R}\to \mathbb{R}$ be a measurable function  satisfying the next assumptions:
\begin{itemize}
	\item[ $(a_{1})$] $\mathcal{A}$ is continuous, odd, and the map $\mathscr{A}:\mathbb{R}\to \mathbb{R}$ given by
	$$\mathscr{A}(t):= \int^{|t|}_{0} \mathcal{A}(\tau) d\tau$$
is  strictly convex;
	\item[$(a_{2})$]  There exist positive constants $c_{\mathcal{A}}$ and $C_{\mathcal{A}}$, such that  for  all $t \in \mathbb{R}$ and  for all $(x,y)\in \mathbb{R}^{N}\times\mathbb{R}^{N}$
\begin{equation*}
\mathcal{A}(t)t\geqslant  c_{\mathcal{A}} \vert t \vert^{p(x,y)} \hspace{0.3cm} \mbox{ and } \hspace{0.3cm} \vert \mathcal{A}(t) \vert\leqslant C_{\mathcal{A}} \vert t \vert^{p(x,y)-1};
\end{equation*}
\item[$(a_{3})$]$\mathcal{A}(t)t\leqslant p^{+}\mathscr{A}(t)$ for all $t \in \mathbb{R}$.
\end{itemize}

The kernel $K: \mathbb{R}^{N}\times \mathbb{R}^{N} \to \mathbb{R}^{+}$ be a measurable function  satisfying the following property:
\begin{itemize}
	\item[ $(\mathcal{K})$] There exist constants $b_0$ and $b_1$, such that $0<b_0\leqslant b_1$, 
	$$b_0\leqslant K(x,y)|x-y|^{N+sp(x,y)}\leqslant b_1 \mbox{ for all } (x,y)\in \mathbb{R}^{N}\times\mathbb{R}^{N}  \mbox{ and } x\neq y.$$
	\end{itemize}

 We assume that $f:\Omega\times \mathbb{R}\to \mathbb{R}$ is a Carath\'eodory function and satisfies:   
 \begin{itemize}
 \item[$(f_{0})$] There exists a positive constant $c_{1}$ such that $f$ satisfies the subcritical growth condition 
$$|f(x,t)|\leqslant c_{1}(1+|t|^{\vartheta(x)-1})$$
for all $(x,t)\in \Omega\times \mathbb{R},$ where $\vartheta\in C(\overline{\Omega}),$ $1<p^{+}<\underline{\vartheta}^{-}\leqslant \vartheta(x)\leqslant \underline{\vartheta}^{+}<p^{\star}_{s}(x)$ for $x\in\overline{\Omega},$ and  $\underline{\vartheta}^{-}:= \displaystyle{\inf _{x \in \overline{\Omega}}}\,\vartheta(x) $,  $\underline{\vartheta}^{+}:= \displaystyle{\sup _{x \in \overline{\Omega}}}\,\vartheta(x); $

 \item[$(f_{1})$] 
 $\displaystyle{\lim_{|t|\to +\infty}\frac{F(x,t)}{|t|^{p^{+}}}=+\infty}$ uniformly for almost everywhere a.e. $x\in \Omega,$ that is, $f$ is $p^{+}$-superlinear at infinity, the function $F$ is the primitive of $f$ with respect to the second variable, that is, $F(x,t):=\displaystyle{\int_{\Omega}f(x,\tau)\,d\tau};$

\item[$(f_{2})$] $f(x,t)= o(|t|^{p^{+}-1}),$ as $t\to 0,$ uniformly a.e. $x\in\Omega;$

\item[$(f_{3})$] There exists a positive constant $\mathsf{c}_{\star}>0$ such that
$$\mathcal{G}(x,t)\leqslant \mathcal{G}(x,\tau)+\mathsf{c}_{\star}$$
for all $x\in \Omega,\,\,0<t<\tau$ or $\tau<t<0,$ where $\displaystyle{\mathcal{G}(x,t):=tf(x,t)-p^{+} F(x,t)}.$
\end{itemize}
\noindent With intending to find infinite solutions is natural to impose certain symmetry condition on the nonlinearity. In the sequel, we will assume the following assumption on $f:$
\begin{itemize}
\item[$(f_{4})$] $f$ is odd in $t,$ that is, 
$f(x,-t)=-f(x,t)$ for all $x\in\Omega$ and $t\in \mathbb{R}$.
\end{itemize}
 
 \noindent In addition, to prove that the Euler Lagrange functional associated to  the problem \eqref{e1.1} verifies the Cerami condition $(C)_c$, we assume that the functions $\mathscr{A} $ and $ \mathcal{A}$ satisfy the following condition
\begin{itemize}
\item[$(a_{4})$]$\mathcal{H}(at)\leqslant \mathcal{H}(t)$ for all $t \in \mathbb{R}$ and $a\in[0,1]$ where $\mathcal{H}(t)=p^{+}\mathscr{A}(t)-\mathcal{A}(t)t.$
\end{itemize}

 A    typical mathematical generalization    for $\mathcal{A}$ and $K$ verifying  $(a_1)$-$(a_4)$ and $(\mathcal{K})$ respectively is $\mathcal{A}(t)=|t|^{p(x,y)-2}t$ and  $K(x,y)= |x-y|^{-(N+sp(x,y))}$.  When  $\mathcal{A}$ and $K$ are these type, the operator $\mathscr{L}_{\mathcal{A}K}$ is the  fractional $p(\cdot)$-Laplacian operator $(-\Delta)^{s}_{p(\cdot)}$, which is defined by
\begin{equation*}
  (-\Delta)^{s}_{p(\cdot)}u(x)= P.V. \int_{\mathbb{R}^{N}}\frac{|u(x)-u(y)|^{p(x,y)-2}(u(x)-u(y))}{|x-y|^{N+sp(x,y)}}\,dy \mbox{ for all } x \in \mathbb{R}^{N}.
  \end{equation*}
For further  details on the fractional $p(\cdot)$-Laplacian operator see for instance \cite{anaour,bahrouni,ky,kaufmann}. Furthermore, it is worth to note that the assumptions $(a_1)$-$(a_4)$ and $(\mathcal{K})$ were similarly introduced  in \cite{mazon,moli,brezish,colau,emdi,warma,elard1,elard2,kim,migio,puccisara,vazquez}.
 
 \noindent In literature problem \eqref{e1.1}  was investigated primarily in the context of equations driven by the Laplacian (semilinear problems), i.e.,   was considered the following  problem   class
 \begin{equation}\label{e1.10}
\left\{\begin{array}{rcll}
-\Delta u &=& \lambda f(x,u) & \textrm{ in } \Omega,\\
u &=& 0 & \textrm{ in }\partial\Omega.
\end{array}\right.
\end{equation}

\noindent Particularly  Ambrosetti and Rabinowitz in  \cite{ar},   were the first to use the Mountain Pass Theorem to prove that problems type above  admits a solution when  the  following condition of the nonlinearity to $f(x, \cdot)$, well known in the literature as the Ambrosetti-Rabinowitz condition ((AR) condition, for short), is employed:  there exist $\mu >2$ and  $M>0$ such that  
\begin{equation}\label{integra}
0< \mu F(x,t) \leqslant f(x,t)t  \mbox{ for all } x \in \Omega \mbox{ and for all } |t|\geqslant M
\end{equation}
where $f:\overline{\Omega}\times \mathbb{R}\to \mathbb{R}$ is a continuous function and $\displaystyle{F(x,t)=\int_{0}^{t}f(t,z)\,dz}$.
The (AR)   condition is   a tool essential to obtain weak solutions of superlinear problems. Moreover,  their main role  is to ensure the compactness, more specifically, the limitation of the Palais-Smale sequence required by minimax arguments. But, the (AR) condition is some restricted and eliminates some nonlinearities of $ F (x, \cdot).$ Direct integration of \eqref{integra} leads us to the following weak condition for the potential function $ F (x, t), $

 \begin{equation} \label{condicao2}
F(x,t) \geqslant \alpha_1|t|^{\mu}-\alpha_2 \mbox{ for  all } x \in \overline{\Omega} \mbox{ and  for all } t \in \mathbb{R} \mbox{ with } \alpha_1, \alpha_2 >0.
\end{equation}

However, even so, there are still many functions that are superlinear at infinity and that not satisfy the (AR) condition.  Indeed, the condition  \eqref{condicao2}   implies  that the function $ F (x, \cdot) $ displays at least one polynomial growth $ \mu $ near to $ \pm \infty $ and since that $ \mu> p$
\begin{equation} \label{condicao3}
\lim_{|t|\to +\infty}\frac{F(x,t)}{|t|^{p}}=+\infty.
\end{equation}
The function $f(x, t) = t\ln(1 +  t) $  is a example   of   an function that satisfies the condition \eqref{condicao3} but does not satisfy the condition \eqref{condicao2} and  therefore also  not satisfies  the condition \eqref{integra}.

In this sense,  the study for problems involving not only Laplacian operator but also $ p(\cdot) $-Laplacian operator without the (AR) condition has become the focus for many researchers and the references in the literature have steadily increased.  We mention some papers as   \cite{ge,elard1,ji,li,zhaoli,shi bo,miyagaki} where the authors established at least nontrivial  solution existence to the problems of type \eqref{e1.10} without (AR) condition and in general, the main tool used to obter the results is Mountain Pass Theorem with Cerami condition.

Furthemore to this significant advance for local operators without the (AR) condition, recently, some researchers have started to study problems of type  \eqref{e1.10} nonlocal  without the (AR) condition.  More specifically, in the paper \cite{zang} considering   $\mathcal{A}(t)=t$, $K(x,y)=|x-y|^{N +2s}$, and $\lambda=1$, the authors proved that the problem \eqref{e1.1} has infinite solutions using Fountain Theorem. In the paper \cite{pei} the author, for the case fractional $p$-Laplacian, $K(x,y)=|x-y|^{N+2s}$ and $\lambda=1$ studies the existence of a weak solution to the problem \eqref{e1.1}  by  Mountain Pass Theorem combined with the fractional Moser–Trudinger inequality.  Already in \cite {zhou1}, for $ \mathcal{A}(t) = t $ and some similar assumptions for the $ K $ kernel, the author proves that $ \lambda_ {0}> 0 $ exists such that the problem  \eqref{e1.1}  supports two distinct weak solutions for every $ \lambda \in (0, \lambda_ {0}) $ using a new critical  points result introduced by \cite[Theorem 2.6] {morse}.

Therefore, motivated by the above references, mainly by the papers \cite{elard1,miyagaki,zang,zhou1},  the  present paper is concerned about with the existence and multiplicity of solutions to the problem \eqref{e1.1} without the (AR) condition.  In particular,   we extend, we complement and improve some of the references cited above once our problem \eqref{e1.1} involves general nonlocal integrodifferential operators. 
Indeed,  is important note that the operator $\mathscr{L}_{\mathcal{A}K} $ involving in problem \eqref{e1.1} is, in particular, the fractional $ p(\cdot)$-Laplacian operator, so it is nonhomogeneous and problems involving this operator with variable exponents can not use Lagrange Multiplier Theorem, as can be done for operators involving constant exponents.  It still is interesting to point  that the  operator $\mathscr{L}_{\mathcal{A}K} $ is no just a mere extension the fractional $ p(\cdot)$-Laplacian operator, because from conditions  $(a_1)$-$(a_4)$ and $(\mathcal{K}),$  map $ \mathcal{A} $ and kernel $ K $ are quite general and  $ K $  includes singular kernels.
 Furthermore,  we note that solution space for the problem \eqref{e1.1} must be a fractional Sobolev space with variable exponents as we can see in \cite{anaour,bahrouni,ky,kaufmann}.  Thus,  due to the characterization of our problem \eqref{e1.1}  it was required to define a new space in this paper.  Then   inspired by the spaces mentioned above and the space $ W^{s, p}_{0}(\Omega) $ (defined in \cite{ian1}), we define the following space
\begin{equation*}\label{space}
  \mathscr{W} = W^{s,p(\cdot, \cdot)}_{0}:= \{u \in W^{s, p(\cdot, \cdot)}(\mathbb{R}^{N}): u=0 \mbox{ a.e. in }   \mathbb{R}^{N}\setminus \Omega \},
 \end{equation*}
  which is a separable and reflexive Banach space  and we prove an important   result of compact and continuous embedding for $\mathscr{W}$, for more details see Subsection \ref{espacop1}.

Moreover, as our problem has a structure variational and our  objective is to prove  existence, multiplicity and comportment of solution of the problem \eqref{e1.1},   is  essential verify  that the operator  $\Phi'$ (see Lemma \ref{ll1}) satisfies the property $(S_{+})$, which is a property of compactness of the operator and usually is used for other property of compactness as for example the condition Palais-Smale or Cerami  for the Euler Lagrange functional $\Psi_{\lambda}$ defined in  \eqref{euler}. 

Lastly, we will  obtain the existence of weak solutions to the problem \eqref{e1.1} and,    under some considerations, get a nontrivial solution for any parameter $\lambda > 0 $ small enough such that the solution blows up, in the fractional Sobolev norm, as $\lambda \to 0^{+}$.  For this,  along of the paper,   we will use different versions of the Mountain Pass Theorem, as well as, the Fountain and Dual Theorem with Cerami condition.   Moreover, by imposing additional hypotheses on the nonlinearity $f (x, \cdot)$,  we get the existence of infinitely many weak solutions which tend to be zero, in the fractional Sobolev norm, for any parameter $\lambda >0$ by using the new version of the symmetric of the Pass Mountain Theorem introduced by Kajikya  
\cite{kaji}.
 
 Our main results are the following Theorems.

\begin{theorem}\label{cc}
Assume $(a_{1})$-$(a_{4}),\,\,(\mathcal{K}),$ and $f$ satisfies $(f_{0})$-$(f_{3})$. Then  problem \eqref{e1.1} has at least one nontrivial weak solution in  $\mathscr{W}$ for all $\lambda>0$.
\end{theorem}
 
 \begin{theorem}\label{t7}
Assume $(a_{1})$-$(a_{4}),\,\,(\mathcal{K}),$ and  $f$ satisfies $(f_{0})$-$(f_{4})$. Then  problem \eqref{e1.1} has infinitely many solutions in $\mathscr{W}$ for all $\lambda>0$.
\end{theorem}

 \begin{theorem}\label{fountain}
Assume $(a_{1})$-$(a_{4}),$ $(\mathcal{K}),$ and  $f$ satisfies $(f_{0})$, $(f_{1})$, $(f_{3})$ and $(f_{4})$. Then, for each $\lambda\in\left(0,\frac{c_{\mathcal{A}}b_{0}}{p^{+}}\right),$ the problem $\eqref{e1.1}$ has infinitely many weak solutions $u_{k} \in \mathscr{W}$, $k \in \mathbb{N}$ such that $\Psi_{\lambda}(u_{k})\to +\infty,$ as $k\to +\infty.$ ($\Psi_{\lambda}$ is defined  \eqref{euler})
\end{theorem}

\begin{theorem}\label{t5}
  Assume $(a_{1})$-$(a_{4}),$ $(\mathcal{K}),$ and  $f$ satisfies $f$ satisfies $(f_{0})$, $(f_{1})$, $(f_{3})$ and $(f_{4}).$ Then, for each $\lambda\in\left(0,\frac{c_{\mathcal{A}}b_{0}}{p^{+}}\right),$ the problem \eqref{e1.1} has a sequence of weak solutions $v_{k} \in \mathscr{W}$, $k\in \mathbb{N}$  such that 
 $\Psi_{\lambda}(v_{k})<0,$  $\Psi_{\lambda}(v_{k})\to 0$ as
 $k\to +\infty.$
\end{theorem}

 \begin{theorem}\label{t6}
Assume $(a_{1})$-$(a_{4})$ and $(\mathcal{K}).$ If $f$ satisfies  $(f_{0})$, $(f_{1})$,  and $(f_{3})$, moreover $f(x,0)=0,\,\,f(x,t)\geqslant 0
\mbox{ a.e. }x\in \Omega$ and for all $t\geqslant 0.$ Then there exists a positive constant $\overline{\lambda}$ such that problem \eqref{e1.1} possesses at least one solution for all $\lambda\in (0,\overline{\lambda})$. Moreover 
$$\lim_{\lambda\to 0^{+}}\|u_{\lambda}\|_{\mathscr{W}}=+\infty.$$
($\|\cdot\|_{\mathscr{W}}$ is defined in Subsection\eqref{espacop1})
\end{theorem}

\begin{theorem}\label{genus}
Assume $(a_{1})$-$(a_{3}),$ $(\mathcal{K})$, and $f$ satisfies  $(f_{4})$. In addition we will assume the following condition:
\begin{itemize}
\item[$(f_{5})$] $f:\Omega\times\mathbb{R}\to \mathbb{R}$ is a continuous function and there exist positive constants $C_{0},$ $C_{1}$ such that
$$C_{0}|t|^{\mathfrak{m}(x)-1}\leqslant f(x,t)\leqslant C_{1}|t|^{\mathfrak{m}(x)-1} \mbox{ for all } x\in \overline{\Omega} \mbox{  and }t\geqslant 0,$$
\end{itemize}
 where $\mathfrak{m}\in C(\overline{\Omega})$ such that $1<\mathfrak{m}(x)<p^{\star}(x)$ for all $x\in \overline{\Omega},$ with $\underline{\mathfrak{m}}^{+}<p^{-}.$ 
Then  problem \eqref{e1.1} it has  a sequence of on trivial solutions $u_{k} \in \mathscr{W}$, $k\in \mathbb{N}$  such that 
$$\lim_{k\to +\infty}\|u_{k}\|_{\mathscr{W}}=0$$
 for all $\lambda>0.$
\end{theorem}
We note that   our results show that the main results of  \cite{elard1,ji,kimar,li,miyagaki,rodrigues,zhou} remain valid for a wider class of nonlocal operators involving variable exponents.
Moreover, as far as we are aware, there are no results on this approach even involving  the fractional $p$-Laplacian problems as well as the fractional $p(\cdot)$-Laplacian problems although we consider some the well-known technique.

 This paper is organized as follows. In Section 2, we give some preliminary results about Lebesgue and fractional Sobolev spaces with variable exponents.  Besides,  we define the function space for the problem \eqref{e1.1}, we give an important result of compact and continuous embedding, and we proof fundamentals properties of Euler Lagrange functional associated with the  of problem \eqref{e1.1}. In Section 3, we prove the main theorems of this paper.

 \section{Preliminaries}\label{sectio2}
 \hfill \break 
In this section, we give some preliminary results which will be used in the sequel to discuss the problem  \eqref{e1.1}.
 Throughout the present paper $c_i$ and $C_i$ for $i=1,2,\ldots $  will denote generic positive constants which
may vary from line to line, but are independent of the terms which take part in any limit process.

\subsection{ Variable exponent Lebesgue space  }
 \hfill \break
In this subsection, we briefly review some of the basic properties of the  Lebesgue  spaces with variable exponent, for more can be found in \cite{alves,dien,edf,fann,fan,radu} and the references therein.

Let  $\Omega \subset \mathbb{R}^{N}$ $(N \geqslant 2)$ an  smooth bounded domain. We introduced

  $$C^{+}(\overline{\Omega})=\big\{h\in C(\overline{\Omega}): h(x)>1\mbox{ for all } x \in \overline{\Omega}\big\}$$
 and for $\sigma \in C^{+}(\overline{\Omega}) $, 
denote
$$ \underline{\sigma}^{-} := \inf_{x\in \overline{\Omega}}\,\sigma(x),  \underline{\sigma}^{+} :=\sup_{x\in \overline{\Omega}}\,\sigma(x).$$
For $\sigma \in C^{+}(\overline{\Omega}), $ we define the variable exponent Lebesgue space $L^{\sigma(\cdot)}(\Omega)$ as
 \begin{equation}\label{lp1}
L^{\sigma(\cdot)}(\Omega):=\Bigg\{u:\Omega \to \mathbb{R} \mbox{  measurable }: \exists\hspace{0.1cm} \zeta>0: \int_{\Omega}\Big|\frac{u(x)}{\zeta}\Big|^{\sigma(x)}\, dx <+\infty \Bigg\}.
\end{equation} 
The space $L^{\sigma(\cdot)}(\Omega)$  endowed with the Luxemburg norm,
$$\|u\|_{L^{\sigma(\cdot)}(\Omega)}:= \inf\Bigg\{\zeta>0:  \int_{\Omega}\Big|\frac{u(x)}{\zeta}\Big|^{\sigma(x)}\, dx\leqslant 1 \Bigg\}.$$
  is a separable and reflexive Banach space.    If $\Omega=\mathbb{R}^{N}$, we simply denote the norm of any $u \in L^{\sigma(\cdot)}(\mathbb{R}^{N}) $ by $\|u\|_{L^{\sigma(\cdot)}(\mathbb{R}^{N})}$. Note that, when  $\sigma$ is constant, the Luxemburg norm $\|\cdot\|_{L^{\sigma(\cdot)}(\Omega)}$ coincide with the standard norm $\|\cdot\|_{L^{\sigma}(\Omega)}$ of the Lebesgue space $L^{\sigma}(\Omega).$
 
 Denoting by $L^{\sigma'(\cdot)}(\Omega)$ the dual space of space $L^{\sigma(\cdot)}(\Omega)$, where $\frac{1}{\sigma(x)}+ \frac{1}{\sigma'(x)}=1$, for any  $u \in L^{\sigma(\cdot)}(\Omega)$ and $v \in L^{\sigma'(\cdot)}(\Omega)$  the following H\" older type inequality holds
 \begin{equation*} 
\int_{\Omega}|uv|\,dx\leqslant \bigg( \frac{1}{\underline{\sigma}^{-}}+\frac{1}{\underline{\sigma}'^{-}}\bigg)\|u\|_{ L^{\sigma(\cdot)}(\Omega)}\|v\|_{ L^{\sigma'(\cdot)}(\Omega)}.
\end{equation*}
If $\sigma_1, \sigma_{2} \in C^{+}(\overline{\Omega})$, $\sigma_{1}(x)\leqslant \sigma_{2}(x)$  for all $x \in\overline{\Omega}$, then $L^{\sigma_2(\cdot)}(\Omega)\hookrightarrow L^{\sigma_1(\cdot)}(\Omega)$, and the embedding is continuous.
 
\noindent An important role in manipulating the generalized Lebesgue spaces is played by the    $\sigma(\cdot)$-modular of the $ L^{\sigma(\cdot)}(\Omega)$ space, which is the convex
function $\rho_{\sigma(\cdot)}: L^{\sigma(\cdot)}(\Omega)\to \mathbb{R}$  defined by 
$$\rho_{\sigma(\cdot)}(u)=\int_{\Omega}|u|^{\sigma(x)}dx,$$
 along any function  $u \in L^{\sigma(\cdot)}(\Omega)$. 
 \begin{proposition}
\label{masmenos}
Let $u\in L^{\sigma(\cdot)}(\Omega) $ and $(u_{k})_{k\in \mathbb{N}}\subset L^{\sigma(\cdot)}(\Omega),$ then the following relations hold: 
\begin{itemize}
  \item[$(a)$]For $u \in L^{\sigma(\cdot)}(\Omega)\setminus \{0 \}$, $\zeta = \|u\|_{L^{\sigma(\cdot)}(\Omega) } $ if and only if $ \rho_{\sigma(\cdot)}\big(\frac{u}{\zeta}\big)=1 $;
    \item[$(b)$]$\|u\|_{L^{\sigma(\cdot)}(\Omega)}\geqslant 1\Rightarrow \|u\|_{L^{\sigma(\cdot)}(\Omega) }^{\underline{\sigma}^{-}}\leqslant \rho_{\sigma(\cdot)}(u)\leqslant \|u\|_{L^{\sigma(\cdot)}(\Omega) }^{\underline{\sigma}^{+}};$
    \item[$(c)$] $\|u\|_{L^{\sigma(\cdot)}(\Omega)}\leqslant 1\Rightarrow \|u\|_{L^{\sigma(\cdot)}(\Omega) }^{\underline{\sigma}^{+}}\leqslant \rho_{\sigma(\cdot)}(u)\leqslant \|u\|_{L^{\sigma(\cdot)}(\Omega) }^{\underline{\sigma}^{-}};$ 
  \end{itemize}  
  Moreover, $ \|u_{k}-u\|_{L^{\sigma(\cdot)}(\Omega)}\to 0 \Leftrightarrow  \rho_{\sigma(\cdot)}(u_k -u)\to 0 \Leftrightarrow  u_{k} \to u$ in measure in $\Omega$ and $\rho_{\sigma(\cdot)}(u_k)\to \rho_{\sigma(\cdot)}(u)$. In particular,  $\rho_{\sigma(\cdot)}$ is continuous in $L^{\sigma(\cdot)}(\Omega)$, and if furthermore $\sigma \in C^{+}(\overline{\Omega})$, then $\rho_{\sigma(\cdot)}$ is weakly lower semi-continuous.
\end{proposition}

 \subsection{The functional space  $\mathscr{W}$ and their properties}\label{espacop1}
 \hfill \break 
  Let   $\Omega\subset\mathbb{R}^{N},$ $N\geqslant2,$ is a smooth bounded domain and $s \in (0,1)$. We consider two variable exponents $q:\overline{\Omega}\to \mathbb{R}$ and  $p:\overline{\Omega}\times\overline{\Omega}\to \mathbb{R}$, where both $q(\cdot)$ and $p(\cdot, \cdot)$ are continuous function. We assume that:  
 \begin{equation}  \label{q10}\tag{$\mathit{q}_{1}$}
\begin{split} 
   p \mbox{ is symmetric, this is, } p(x,y)=p(y,x),&\\ 1<\underline{p}^{-}:=\inf_{(x,y)\in \overline{\Omega}\times\overline{\Omega} }\,p(x,y)\leqslant p(x,y)\leqslant\sup_{(x,y)\in \overline{\Omega}\times\overline{\Omega}}\,p(x,y):=\underline{p}^{+}<+\infty,&
   \end{split}
\end{equation}
 and
 \begin{equation}\label{qq1} \tag{$\mathit{q}_{2}$}
    1<\underline{q}^{-}:=\inf_{x\in \overline{\Omega} }\,q(x)\leqslant q(x)\leqslant\sup_{x\in \overline{\Omega}}\,q(x):=\underline{q}^{+}<+\infty.
\end{equation}
The fractional Sobolev space with variable
exponents is defined by
\begin{equation*}
\begin{split}
    W^{s,q(\cdot),p(\cdot, \cdot)}(\Omega) 
    &:=\Bigg\{u\in L^{q(\cdot)}(\Omega): \int_{\Omega \times \Omega}\frac{|u(x)-u(y)|^{p(x,y)}}{\zeta^{p(x,y)}|x-y|^{N+sp(x,y)}}\, dx\,dy< +\infty, \mbox{ for some } \zeta>0 \Bigg\} 
\end{split}
\end{equation*}
 and we set
$$[u]^{s,p(\cdot,\cdot)}_{ \Omega}= \inf\Bigg\{ \zeta>0 ; \int_{\Omega\times \Omega}\frac{|u(x)-u(y)|^{p(x,y)}}{\zeta^{p(x,y)}|x-y|^{N+sp(x,y)}} \,dx\,dy \leqslant 1 \Bigg\}$$
 the variable exponent Gagliardo seminorm.    It is already known  that
$W^{s,q(\cdot),p(\cdot, \cdot)}(\Omega)$ is a separable and reflexive Banach space with the norm
 $$ \|u\|_{W^{s,q(\cdot), p(\cdot,\cdot)}(\Omega)}:=\|u\|_{L^{q(\cdot)}(\Omega)}+[u]_{\Omega}^{s,p(\cdot,\cdot)},$$ see \cite{azr,bahrouni,kaufmann}.

 \begin{remark}
 Throughout this paper, when $q(x)=p(x,x)$ we denote $p(x)$ instead of $p(x, x)$ and  we will write $W^{s, p(\cdot, \cdot)}(\Omega)$ instead of $W^{s,p(\cdot), p(\cdot, \cdot)}(\Omega).$
 \end{remark}
The next  result is an consequence of  \cite[Theorem 3.2]{ky}.

 
\begin{corollary}\label{3.4a}
Let $\Omega\subset\mathbb{R}^{N}$  a bounded Lipschitz domain, $s\in(0,1)$, $p(x,y)$ and  $p(x)$ be continuous variable exponents such that \eqref{q10}-\eqref{qq1} be satisfied and that $ s\underline{p}^{+} < N$. Then, for all $r :\overline{\Omega}\rightarrow(1,+\infty)$  a continuous function such that
$p^{\star}_{s}(x)>r(x)$  for all $ x \in \overline{\Omega}$, 
 the space $W^{s, p(\cdot,\cdot)}(\Omega)$ is continuously and compactly embedding  in $L^{r(\cdot)}(\Omega)$.
\end{corollary} 
Now, we consider the space
\begin{equation*}
\begin{split}
  W^{s, p(\cdot,\cdot)}(\mathbb{R}^{N}):= \Bigg\{u\in L^{p(\cdot)}(\mathbb{R}^{N}): \int_{\mathbb{R}^{N} \times \mathbb{R}^{N}}\frac{|u(x)-u(y)|^{p(x,y)}}{\zeta^{p(x,y)}|x-y|^{N+sp(x,y)}} \,dx\,dy < +\infty, \mbox{ for some } \zeta>0 \Bigg\} 
\end{split}
\end{equation*}
where the space ${L^{p(\cdot)}(\mathbb{R}^{N})}$ is defined analogous the space ${L^{p(\cdot)}(\Omega)}$. The corresponding norm for this space  is
$$ \| u \|:= \| u \|_{L^{p(\cdot)}(\mathbb{R}^{N})} + [u]^{s, p(\cdot, \cdot)}_{\mathbb{R}^{N}}.$$ The space $ (W^{s, p(\cdot,\cdot)}(\mathbb{R}^{N}), \|\cdot\|)$  has the same properties that  $( W^{s, p(\cdot,\cdot)}(\Omega), \|\cdot\|_{W^{s, p(\cdot,\cdot)}(\Omega)})$, this is, it is  a  reflexive and separable Banach space.
 
 \noindent We define the space were will study the problem \eqref{e1.1}. Let we will consider   the variable exponents  $p(x):=p(x,x)$ for all $x \in \mathbb{R}^{N}$ with   $p \in C(\mathbb{R}^{N}\times \mathbb{R}^{N})$ satisfying \eqref{a23}  and we denote by  
 \begin{equation*}\label{space}
  \mathscr{W} = W^{s,p(\cdot, \cdot)}_{0}:= \{u \in W^{s, p(\cdot, \cdot)}(\mathbb{R}^{N}): u=0 \mbox{ a.e. in }   \mathbb{R}^{N}\setminus \Omega \}.
 \end{equation*}
Note that $\mathscr{W} $ is  a   closed subspace of $W^{s, p(\cdot, \cdot)}(\mathbb{R}^{N})$,  thus  $ \mathscr{W}$ is a reflexive and  separable Banach space with the norm
$$ \| u \|:= \| u \|_{L^{p(\cdot)}(\Omega)} + [u]^{s, p(\cdot, \cdot)}_{\mathbb{R}^{N}},$$ once the norms $\|\cdot\|_{L^{p(\cdot)}(\mathbb{R}^{N})}$  and $\|\cdot\|_{L^{p(\cdot)}(\Omega)}$  coincide in $ \mathscr{W} $.

 The next Lemma shows that  the space $(\mathscr{W}, \|\cdot\|) $ is equivalently defined  with respect to the Gagliardo seminorm 
$[\cdot]^{s, p(\cdot, \cdot)}_{\mathbb{R}^{N}}.$ See proof in Appendix \ref{apendice}.

\begin{lemma}\label{lw1}
Assume $\Omega$ be a smooth bounded domain in $\mathbb{R}^{N}$. Let  $p(x):=p(x,x)$ for all $x \in \mathbb{R}^{N}$ with $p \in C(\mathbb{R}^{N}\times \mathbb{R}^{N})$ satisfying \eqref{a23}  and $p^{\star}_{s}(x) > p(x)$ for $x \in \mathbb{R}^{N}$.   Then there exists $ \zeta_1>0$ such that
\begin{equation*}\label{w1}
\|u\|_{L^{p(\cdot)}(\Omega)}\leqslant \frac{1}{\zeta_1}[u]^{s, p(\cdot, \cdot)}_{\mathbb{R}^{N}}\mbox{ for all } \hspace{0.1cm} u \in \mathscr{W}.
\end{equation*}
\end{lemma}

\noindent Therefore, we  will consider the space  $\mathscr{W}$ with norm $\|u\|_{\mathscr{W}}= [u]^{s, p(\cdot, \cdot)}_{\mathbb{R}^{N}}$ and we proof  an important result of compact and continuous embedding of space $\mathscr{W}$ as consequence of Corollary \ref{3.4a} and  Lemma \ref{lw1}.  See proof of Lemma \ref{2.11} in  Appendix \ref{apendice}.

\begin{lemma}\label{2.11}
Assume $\Omega$ be a smooth bounded domain in $\mathbb{R}^{N}$. Let $p(x):=p(x,x)$ for all $x \in \mathbb{R}^{N}$ with  $p \in C(\mathbb{R}^{N}\times \mathbb{R}^{N})$ satisfying \eqref{a23}  and $p^{\star}_{s}(x) > p(x)$ for $x \in \mathbb{R}^{N}$. 
Assume that $r:\overline{\Omega}\rightarrow (1, +\infty)$ is a continuous function. Thus, the space $(\mathscr{W}, \|\cdot\|_{\mathscr{W}})$  is continuously and compactly embedding   in $L^{r(\cdot)}(\Omega)$ for all $r(x) \in (1, p^{\star}_{s}(x))$ for all $x \in \overline{\Omega}$.
\end{lemma}

   An important role in manipulating the fractional Sobolev spaces with variable exponent is played by the $(s, p(\cdot, \cdot))$-convex modular function $\rho_{\mathscr{W}}: \mathscr{W} \to \mathbb{R}$ defined by
  $$\displaystyle{\rho_{\mathscr{W}}(u)= \int_{\mathbb{R}^{N}\times \mathbb{R}^{N}}\frac{|u(x)-u(y)|^{p(x,y)}}{|x-y|^{N+sp(x,y)}} \,dx\,dy}.$$
  The following proposition show  the relationship between the norm $\|\cdot\|_{\mathscr{W}}$   and the $\rho_{\mathscr{W}}$ convex modular function.
 
 \begin{proposition}\label{lw0}
For $u\in \mathscr{W}$ and $(u_{k})_{k\in \mathbb{N}}\subset\mathscr{W}$, we have
\begin{itemize}
\item[$(a)$]For $u \in \mathscr{W}\setminus \{0 \}$, $\zeta = \|u\|_{\mathscr{W} } $ if and only if $ \rho_{\mathscr{W}}\big(\frac{u}{\zeta}\big)=1 $;
    \item[$(b)$]$\|u\|_{  \mathscr{W}}\geqslant 1\Rightarrow \|u\|_{ \mathscr{W}}^{p^{-}}\leqslant \rho_{ \mathscr{W}}(u)\leqslant \|u\|_{  \mathscr{W} }^{p^{+}};$
    \item[$(c)$] $\|u\|_{  \mathscr{W}}\leqslant 1\Rightarrow \|u\|_{  \mathscr{W}}^{p^{+}}\leqslant \rho_{ \mathscr{W}}(u)\leqslant \|u\|_{  \mathscr{W} }^{p^{-}};$ 
    \item[$(d)$] $\lim\limits_{k\to+\infty} \|u_{k}-u\|_{  \mathscr{W}}=0 \Leftrightarrow \lim\limits_{k\to+\infty} \rho_{ \mathscr{W}}(u_{k}-u)=0;$
 \item[$(e)$] $\lim\limits_{k\to+\infty} \|u_{k}\|_{  \mathscr{W}}=+\infty \Leftrightarrow \lim\limits_{k\to+\infty} \rho_{ \mathscr{W}}(u_{k})=+\infty.$
\end{itemize}
\end{proposition}
 
 \subsection{Auxiliary Result}
  \hfill \break

 \begin{lemma} \label{ll1}
Assume that  $(a_{1})$-$(a_{3})$, and $(\mathcal{K})$ is  hold. We consider the functional $\Phi:\mathscr{W}\to \mathbb{R}$ defined by
\begin{equation*}
\Phi(u)= \int_{\mathbb{R}^{N}\times \mathbb{R}^{N}}\mathscr{A}(u(x)-u(y))K(x,y)\,dx\,dy  \mbox{ for all }\hspace{0.1cm} u \in \mathscr{W},
\end{equation*}
has the following  properties:
\begin{itemize}
\item[$(i)$] The functional $\Phi$ is well defined on $\mathscr{W}$, is of   class  $C^{1}(\mathscr{W}, \mathbb{R})$, and its G\^ateaux derivative is given by
\begin{equation*}\label{phi'}
\langle \Phi'(u), v \rangle = \int_{\mathbb{R}^{N}\times \mathbb{R}^{N}}\mathcal{A}(u(x)-u(y))(v(x)-v(y))K(x,y)\,dx\,dy \mbox{ for all } \hspace{0.1cm} u, v \in \mathscr{W};
\end{equation*} 
\item[$(ii)$]  The functional $\Phi$ is weakly lower semicontinuous, that is, $u_k \rightharpoonup u$ in $\mathscr{W}$ as $ k \to +\infty$ implies that $\displaystyle{\Phi(u) \leqslant\liminf_{k\to +\infty} \Phi(u_k)}$; 
\item[$(iii)$] The  functional $\Phi' : \mathscr{W}\to \mathscr{W}'$ is an operator of type $(S_{+})$ on $\mathscr{W}$, that is, if 
  \begin{equation} \label{inffo}
  u_k \rightharpoonup u  \mbox{ in }\mathscr{W} \mbox{ and } \limsup_{k \to +\infty}\,\langle \,\Phi'(u_k), u_k-u \rangle\leqslant 0,
  \end{equation}
  then $u_k\to u$ in $\mathscr{W}$ as $k\to +\infty$.
\end{itemize}
\end{lemma}
 The proof of the  result above    will be referred to in Appendix \ref{apendice}.
 
 \begin{definition}
We say that $u\in  \mathscr{W}$ is a weak solution to the problem \eqref{e1.1} if, and only if,
$$\int_{\mathbb{R}^{N}\times \mathbb{R}^{N}}\mathcal{A}(u(x)-u(y))(v(x)-v(y))K(x,y)\,dx\,dy=\lambda\int_{\Omega}f(x,u)v\,dx \mbox{ for all }  \hspace{0.1cm} v\in \mathscr{W}.$$
\end{definition}

\noindent The weak solutions to the problem  \eqref{e1.1} coincide with the critical points of the Euler Lagrange functional $\Psi_{\lambda}:\mathscr{W} \to \mathbb{R}$ given by
\begin{equation}\label{euler}
\Psi_{\lambda}(u)=\Phi(u)-\lambda\int_{\Omega}F(x,u)\,dx \mbox{ for all } u\in \mathscr{W}.
\end{equation} 
 Moreover,   by Lemma \ref{ll1} and standard arguments,  the functional $\Psi_{\lambda}$ is Fréchet differentiable in $u\in \mathscr{W}$  and
\begin{equation*}
\langle \Psi'_{\lambda}(u), v \rangle = \int_{\mathbb{R}^{N}\times \mathbb{R}^{N}}\mathcal{A}(u(x)-u(y))(v(x)-v(y))K(x,y)\,dx\,dy -\lambda\int_{\Omega}f(x,u)v\,dx \mbox{ for all }  v \in \mathscr{W}.
\end{equation*}

  \subsection{ Standard results for the existence and multiplicity of weak solutions}
    \hfill \break
  In this subsection, we enunciated some  definitions and general theorems of the existence and multiplicity of weak solutions that we will use along of the paper to prove our main results, stated in Theorem \ref{cc}, Theorem \ref{t7}, Theorem \ref{fountain}, Theorem \ref{t5}, Theorem \ref{t6}, and Theorem \ref{genus}.

 \begin{definition}
Let $X$ be a real Banach space and $\Psi\in C^{1}(X,\mathbb{R}).$ We say that $\Psi$ satisfies the $(C)_{c}$ condition if for every sequence $(u_{k})_{k\in \mathbb{N}}\subset X$ such that $\Psi(u_{k})\to c$ and $\|\Psi'(u_{k})\|_{X'}(1+\|u_{k}\|_{X})\to 0,$ as $k\to +\infty,$ has a convergent subsequence.
\end{definition}
 The condition  $(C)_{c}$,  introduced by Cerami in \cite{cerami1,cerami2}, is a little more weak version of the Palais–Smale $(PS)_{c}$ condition, acondition more common that we find in the literature. Thus, since the Deformation Theorem is still valid under the Cerami a condition it follows that the Mountain Pass Theorem, Fountain Theorem, and Dual Fountain Theorem under Palais–Smale $(PS)_{c}$ and $(PS)^{\star}_{c}$ condition holds true also under this compactness.

 The results below Mountain Pass Theorem, $"\mathbb{Z}_{2}- symmetric "$ version (for even functionals) Mountain Pass Theorem, Fountain Theorem and   Dual Fountain Theorem    can be seen respectively in \cite[Theorem I]{MAGA}, \cite[Theorem 9.12]{rab},  \cite[Theorem 2.9]{shi bo}, and \cite[Theorem 2]{barst}.

\begin{theorem} \label{MountainPass}
   Let $X$ be a real Banach space, let $\Psi:X\to \mathbb{R}$ 
   be a functional of class $C^{1}(X,\mathbb{R})$ that satisfies the 
   $(C)_{c}$ condition for any   
   $c>0,$ $\Psi(0)=0,$ and the following conditions hold:
\begin{itemize}

\item[(i)] There exist positive constants $\rho$ and $\mathcal{R}$ such that $\Psi(u)\geqslant \mathcal{R}$ for any $u\in X$ with $\|u\|_{X}=\rho;$

\item[(ii)] There exists a function $e\in X$ such that $\|e\|_{X}>\rho$ and $\Psi(e)<0.$
\end{itemize} 
Then, the functional $\Psi$ has a critical value $c\geqslant \mathcal{R},$ that is, there exists $u\in X$ such that $\Psi(u)=c$ and $\Psi'(u)=0$ in 
$X'.$ 
\end{theorem}

\begin{theorem}\label{psm}
 Assume that $X$ has infinite dimension and let $\Psi\in C^{1}(X, \mathbb{R})$ be a functional satisfying the  $(C)_{c}$ condition as well as the following properties
 \begin{itemize}
 \item[$i)$]   $\Psi(0)=0$, and there exist two constants $r, \rho> 0$ such that $\Psi_{|_{\partial B_{r}}} \geqslant \rho$ ;
 \item[$ii)$] $\Psi$ is even;
 \item[$iii)$] For all finite dimensional subspace $\widehat{X} \subset  X$ there exists $\mathcal{R} = \mathcal{R}(\widehat{X}) > 0$ such that 
 \begin{equation*}
 \Psi(u)\leqslant 0 \mbox{ for all } u \in \widehat{X} \setminus B_{\mathcal{R}}(\widehat{X})
 \end{equation*}
 where $B_{\mathcal{R}}(\widehat{X})= \{u \in \widehat{X}: \|u\|_{X}< \mathcal{R}\}$.
 \end{itemize}
 Then $\Psi$ possesses an unbounded sequence of critical values.
 \end{theorem}
 
\noindent Let  $X$ be a real, reflexive, and separable Banach space, it is known (\cite[Chapter 4]{fabian}  or \cite[Section 17]{zhao} or \cite{schauder}) that for a separable and reflexive Banach space there exist sequence $(e_l)_{l \in \mathbb{N}}\subset X$ and $(e^{\star}_{l})_{\mathbb{N}}\subset X^{\star}$ such that
$$X=\overline{span\{e_{l}:l=1,2,\ldots\}},
\; X^{\star}=\overline{span\{e_{l}^{\star}:l=1,2,\ldots\}}^{\omega^{\star}},$$
and
$$ \langle e_{i}^{\star},e_{l}\rangle
= \left\{\begin{array}{rll}
1 & \hbox{se} & i=l, \\
0 & \hbox{se} & i\neq l.
\end{array}\right.$$

\noindent We denote
$$X_{l}  =span\{e_{l}\},\; Y_{j}=\bigoplus_{l=1}^{j}X_{l}=span\{e_{1},\ldots,e_{j}\},\mbox{ and } Z_{j}=\overline{\bigoplus_{j=l}^{\infty}X_{l}}=\overline{span\{e_{j},e_{j+1}\ldots\}}.$$  
 
 \begin{theorem}\label{foun}
Assume
\begin{itemize}
\item[$(h_1)$] $X$ is a Banach space, $\Psi\in C^{1}(X,\mathbb{R})$ is an even functional;\\
If for every $j\in \mathbb{N}$ there exit $\rho_{j}>r_{j}>0$ such that
\item[$(h_2)$]$b_{j}:=\inf \{ \Psi(u):u\in Z_{j},\|u\|_{X}=r_{j}\}\to +\infty\,\,as\,\,j\to +\infty;$
\item[$(h_3)$]$a_{j}:=\sup\{\Psi(u):u\in Y_{j},\|u\|_{X}=\rho_{j}\}\leqslant  0;$
\item[$(h_4)$]$\Psi$ satisfies the $(C)_{c}$ condition for every $c>0.$
\end{itemize}
Then $\Psi$ has a sequence of critical points $(u_j)_{j \in \mathbb{N}}$ such that $\Psi(u_{j})\to +\infty$.
\end{theorem}

\begin{definition} Let $X$ be a separable and reflexive Banach space, $\Psi\in C^{1}(X,\mathbb{R}),$ $c\in\mathbb{R}$. We say that $\Psi$ satisfies the $(C)_{c}^{\star}$ condition (with respect to $Y_{k}$), if any sequence $(u_{k})_{k\in\mathbb{N}}\subset X$ for which $u_{k}\in Y_{k},$ for any $k\in\mathbb{N},$ 
$\Psi(u_{k})\to c$ and $\|\Psi'_{|_{Y_{k}}}(u_{k})\|_{X'}(1+\|u_{k}\|_{X})\to 0,$ as $k\to\infty$, contain a subsequence converging to a critical point of $\Psi.$
\end{definition}
 \begin{theorem}\label{dual}
Suppose $(h_1).$ If for each $j\geqslant j_{0}$ there exist $\rho_{j}>r_{j}>0$ such that
\begin{itemize}
\item[$(g_{1})$] $a_{j}=\inf\{\Psi(u):u\in Z_{j},\,\,\|u\|_{X}=\rho_{j}\}\geqslant 0;$
\item[$(g_{2})$] $b_{j}=\sup\{\Psi(u):u\in Y_{j},\,\,\|u\|_{X}=r_{j}\}< 0;$
\item[$(g_{3})$] $d_{j}=\inf\{\Psi(u):u\in Z_{j},\,\,\|u\|_{X}\leqslant \rho_{j}\}\to 0,$ as $j\to+\infty;$
\item[$(g_{4})$] $\Psi$ satisfies the $(C)_{c}^{\star}$ condition for every $c\in [d_{j_{0}},0[.$
\end{itemize}
Then $\Psi$ has a sequence of negative critical values converging to $0$.
\end{theorem}

We conclude this subsection defining the notion of Krasnoselskii  genus and introducing a critical point theorem related to the a new version of the symmetric Mountain Pass Theorem studied by  Kajikiya, see \cite[Theorem 1]{kaji}.

\begin{definition}
   Let $X$ be a real Banach space and $B$ a subset of $X$. $B$ is said to be\textit{ symmetric } if $u\in B$ implies $-u\in B.$ For a closed symmetric set $B$ which does not contain the origin, we define a   \textit{Krasnoselskii genus}  $\gamma(B)$ of $B$  by the smallest integer $j$ such that there exists an odd continuous mapping from $B$ to $\mathbb{R}^{j} \backslash\{0\}.$  If there does not exist such a $j$, we define $\gamma(B)=+\infty.$  Moreover, we set $\gamma(\emptyset)=0.$
   \end{definition}
 
Let us consider the following set,
\begin{equation*}
\Gamma_{j}=\{ B_{j} \subset X:  B_{j} \mbox{ is closed, symmetric and } 0 \notin B_j \mbox{ such that the genus }\gamma(B_{j}) \geqslant j\}
\end{equation*}



\begin{theorem}\label{kajikiya}
 Let $X$ be an infinite-dimensional space, $B\in \Gamma_{j}$, and  $\Psi \in C^{1}(X, \mathbb{R})$  satisfying  the following conditions 
\begin{itemize}
 \item[$(I1)$] $\Psi(u)$ is even, bounded from below, $\Psi(0)=0$ and $\Psi(u)$ satisfies the Palais-Smale condition;
\item[$(I2)$] For each $j \in \mathbb{N}$, there exists an $B_{j} \in \Gamma_{j}$ such that $\displaystyle{\sup_{u\in B_{j}}\Psi(u) < 0}$.
 \end{itemize}
 Then either $(R1)$ or $(R2)$ below holds
 \begin{itemize}
 \item[$(R1)$]  There exists a sequence $(u_{j})_{j \in \mathbb{N}}$ such that $\Psi'(u_{j})=0$, $\Psi(u_{j})<0$  and $(u_{j})_{j \in \mathbb{N}}$
converges to zero.
\item[$(R2)$]  There exist two sequences $(u_{j})_{j \in \mathbb{N}}$  and $(v_{j})_{j \in \mathbb{N}}$  such that $\Psi'(u_{j})=0$, $\Psi(u_{j})<0$, $u_j\neq 0$, $\displaystyle{\lim_{j\to +\infty}}u_j=0$,  $\Psi'(v_{j})=0$, $\Psi(u_{j})<0$, $\displaystyle{\lim_{j\to +\infty}}v_j=0$, and
$(v_{j})_{j \in \mathbb{N}}$ converges to a non-zero limit.
\end{itemize}
 \end{theorem}
 \begin{remark} From Theorem \ref{kajikiya} we have a sequence $(u_{j})_{j \in \mathbb{N}}$  of critical points such that $\Psi(u_{j})\leqslant0$, $u_j\neq 0$,  and $\displaystyle{\lim_{j\to +\infty}}u_j=0$.
\end{remark}
 


 \section{Proof of Theorems \ref{cc}, \ref{t7},  \ref{fountain}, \ref{t5}, \ref{t6}, and  \ref{genus}}
 
 \subsection{Proof of Theorem \ref{cc}}
  \hfill \break
To prove  Theorem \ref{cc} we need of  Lemmas  below and Theorem \ref{MountainPass}.
 \begin{lemma}\label{lema 1}
Assume $(a_{1})$-$(a_{3}),$ $(\mathcal{K}),$ and $(f_{0})$-$(f_{2})$ are holds. Then we have the following assertions:
\begin{itemize}
\item[(i)] There exists $v\in \mathscr{W},$ $v>0,$ such that $\Psi_{\lambda}(tv)\to -\infty$ as $t\to +\infty;$
\item[(ii)] There exist $r>0$ and $\mathcal{R}>0$ such that $\Psi_{\lambda}(u)\geqslant \mathcal{R}$ for any $u\in \mathscr{W}$ with $\|u\|_{\mathscr{W}}=r.$
\end{itemize}
\end{lemma}
 \begin{proof}

(i) From $(f_{1})$, it follows that for any $\mathscr{C}>0$ there exists a constant $
c_{\mathscr{C}}>0$ such that
\begin{equation}\label{1l1}
F(x,t)\geqslant \mathscr{C}|t|^{p^{+}}-c_{\mathscr{C}}\,\,\mbox{ for all } \hspace{0.1cm} (x,t)\in   \Omega \times \mathbb{R}.
\end{equation}
Take $v \in \mathscr{W}$ with $v>0,$ for $t>1$, by (\ref{1l1}), $(a_{1})$, $(a_{2})$, and $(\mathcal{K})$, we obtain
\begin{equation}\label{d1}
\begin{split}
\Psi_{\lambda}(tv)\leqslant  t^{p^{+}}\bigg[  \frac{C_{\mathcal{A}}b_1}{p^{-}} \int_{\mathbb{R}^{N}\times \mathbb{R}^{N}}\frac{|v(x)-v(y)|^{p(x,y)}}{|x-y|^{N+sp(x,y)}}\,dx\,dy -\lambda \mathscr{C}\int_{\Omega}v^{p^{+}}\,dx\bigg] + \lambda c_{\mathscr{C}}|\Omega|
\end{split}
\end{equation}
 where $|\Omega|$ denotes the Lebesgue measure of $\Omega.$ Hence, from \eqref{d1} and taking $\mathscr{C}$ large enough such that
$$ \frac{C_{\mathcal{A}}b_1}{p^{-}} \int_{\mathbb{R}^{N}\times \mathbb{R}^{N}}\frac{|v(x)-v(y)|^{p(x,y)}}{|x-y|^{N+sp(x,y)}}\,dx\,dy -\lambda \mathscr{C}\int_{\Omega}v^{p^{+}}\,dx <0,$$
we have
$$\lim_{t\to+\infty}\Psi_{\lambda}(tv)=-\infty,$$
which completes the proof of $(i).$
\\
\noindent(ii)First,  since the embeddings $\mathscr{W} \hookrightarrow L^{p^{+}}(\Omega)$ and $\mathscr{W} \hookrightarrow L^{\vartheta(x)}(\Omega)$ are continuous (Lemma \ref{2.11}), there exist positive constants $c_{2}, c_{3}$, such that
\begin{equation}\label{3}
\|u\|_{L^{p^{+}}(\Omega)}\leqslant c_{2}\|u\|_{\mathscr{W}},\,\,\|u\|_{L^{\vartheta(\cdot)}(\Omega)}\leqslant c_{3}\|u\|_{\mathscr{W}}\,\,\mbox{ for all } \hspace{0.1cm} u\in \mathscr{W}.
\end{equation}
Now, let $\displaystyle{0<\varepsilon<\frac{c_{\mathcal{A}}b_{0}}{\lambda c_{2}^{p^{+}}}}$. 
From $(f_{0})$ and $(f_{2}),$ it follows that for all given $\varepsilon>0,$ there exists $c_{\varepsilon}>0,$ such that
\begin{equation}\label{4}
F(x,t)\leqslant \frac{\varepsilon}{p^{+}}|t|^{p^{+}}+c_{\varepsilon}|t|^{\vartheta(x)} \, \, \mbox{ for all } \hspace{0.1cm}(x,t)\in \Omega\times \mathbb{R}.
\end{equation}
Now,
Thus, for $u\in \mathscr{W}$ with $\|u\|_{\mathscr{W}}<1$ sufficiently small, from   $(a_{2}),$ $(a_{3}),$ $(\mathcal{K}),$ \eqref{3}, and \eqref{4},  we obtain
\begin{equation}\label{5}
\begin{split}
\Psi_{\lambda}(u)  \geqslant \frac{\|u\|^{p^{+}}_{\mathscr{W}}}{p^{+}}\left(c_{\mathcal{A}}b_{0}-\lambda\varepsilon c_{2}^{p^{+}}\right)-\lambda c_{\varepsilon}c_{3}^{\underline{\vartheta}^{-}}\|u\|^{\underline{\vartheta}^{-}}_{\mathscr{W}}.
\end{split}
\end{equation}
Therefore, since $\underline{\vartheta}^{-}>p^{+}$ from $(\ref{5})$ we can choose $\mathcal{R}>0$ and $r>0$ such that $\Psi_{\lambda}(u)\geqslant \mathcal{R}>0$ for every $u\in \mathscr{W}$ and $\|u\|_{\mathscr{W}}=r.$ This completes the proof of $(ii).$ 
\end{proof}
 
\begin{lemma}\label{cerami}
Assume that the condition $(a_{1})$-$(a_{4}),\,\,(\mathcal{K})$, and $f$ satisfies $(f_{0})$, $(f_{1})$ and $(f_{3})$. Then the functional $\Psi_{\lambda}$ satisfies the $(C)_{c}$ condition at any level $c>0.$
\end{lemma}
\begin{proof}
Let $c\in \mathbb{R}$ and $(u_{k})_{k\in\mathbb{N}}\subset \mathscr{W}$ be a $(C)_{c}$ sequence for $\Psi_{\lambda},$ that is,
\begin{equation}\label{6}
   \Psi_{\lambda}(u_{k})\to c>0 \mbox{ and } 
   \|\Psi'_{\lambda}(u_{k})\|_{\mathscr{W}'}(1+\|u_{k}\|_{\mathscr{W}})\to 0 \mbox{ as } k\to+\infty.
\end{equation}
Initially, we prove that the sequence $(u_{k})_{k\in\mathbb{N}}$ is bounded in $\mathscr{W}.$ Indeed, arguing by contradiction, up to a subsequence, still denoted by $(u_{k})_{k\in\mathbb{N}},$ we suppose that $(u_k)_{k\in\mathbb{N}}$ is unbounded in $\mathscr{W}.$
Define $\displaystyle{\omega_{k}:=\frac{u_{k}}{\|u_{k}\|_{\mathscr{W}}}}$ for all $k\in \mathbb{N},$ then $(\omega_{k})_{k\in\mathbb{N}}\subset \mathscr{W}$ and $\|\omega_{k}\|_{\mathscr{W}}=1.$ Thus, we can extract a subsequence, still denoted by $(\omega_{k})_{k\in\mathbb{N}}$ and $\omega\in \mathscr{W}$ such that $\omega_{k} \rightharpoonup \omega$  in $\mathscr{W}$ $ \mbox{ as } k\to+\infty$. From   Lemma \ref{2.11}, it follows that 
\begin{eqnarray}\label{conver1}
\omega_{k}(x)\to \omega(x) \mbox{ a.e. }  x\in  \Omega, \hspace{0.1cm} \omega_{k}\to \omega \mbox{ in } L^{p^{+}}(\Omega), \hspace{0.1cm} \mbox{ and }\omega_{k}\to \omega \mbox{ in }  L^{\vartheta(\cdot)}(\Omega)  \mbox{ as } k\to+\infty.
\end{eqnarray}
We consider $\Omega_{\star}:=\{x\in \Omega:\omega(x)\neq 0\}.$ If $x\in \Omega_{\star},$ by $(\ref{conver1}),$ we have
\begin{equation*}\label{7}
|u_{k}(x)|=|\omega_{k}(x)|\|u_{k}\|_{\mathscr{W}}\to +\infty \mbox{ a.e. } x\in  \Omega_{\star} \mbox{ as } k\to +\infty.
\end{equation*}
Therefore, by $(f_{1}),$ we obtain for each $x\in \Omega_{\star}$   
\begin{equation}\label{8}
  \lim_{k\to+\infty}\frac{F(x,u_{k})}{|u_{k}|^{p^{+}}}
  \frac{|u_{k}|^{p^{+}}}{\|u_{k}\|^{p^{+}}_{\mathscr{W}}}=\lim_{k\to +\infty}\frac{F(x,u_{k})}{|u_{k}|^{p^{+}}}|\omega_{k}|^{p^{+}}=+\infty.
\end{equation}
Also, by $(f_{1})$ there exists $D>0$ such that
\begin{equation}\label{9}
\frac{F(x,t)}{|t|^{p^{+}}}>1\,\, \mbox{ for all } \hspace{0.1cm}(x,t)\in \Omega\times\mathbb{R} \mbox{ with } |t|\geq D.
\end{equation}
 Since $F(x,t)$ is continuous on $\overline{\Omega}\times [-D,D],$ there exists a positive constant $c_{5}$ such that
\begin{equation}\label{10}
|F(x,t)|\leqslant c_{5}\,\, \mbox{ for all } \hspace{0.1cm} (x,t)\in \overline{\Omega}\times  [-D,D].
\end{equation}
 Hence, by $(\ref{9})$ and $(\ref{10}),$ we conclude that  there is a constant $c_{6}$ such that
\begin{equation*}
F(x,t)\geqslant c_{6}\,\,\mbox{ for all } \hspace{0.1cm}(x,t)\in \overline{\Omega}\times \mathbb{R},
\end{equation*}
which shows that
$$\frac{F(x,u_{k})-c_{6}}{\|u_{k}\|_{\mathscr{W}}^{p^{+}}}\geqslant 0 \,\,\mbox{ for all  } x\in \Omega\mbox{ and  } \ k\in \mathbb{N},$$
that is,
\begin{equation}\label{12}
\frac{F(x,u_{k})}{|u_{k}|^{p^{+}}}|\omega_{k}|^{p^{+}}-\frac{c_{6}}{\|u_{k}\|_{\mathscr{W}}^{p^{+}}}\geqslant 0 \,\,\mbox{ for all  } x\in \Omega \mbox{ and  }  k\in \mathbb{N}. 
\end{equation}
Now, by $(\ref{6})$, $(a_{2})$, $(a_{3})$, and $(\mathcal{K})$, we have
\begin{equation*}
\begin{split}
c    \geqslant& \frac{c_{\mathcal{A}}b_0 }{p^{+}}\int_{\mathbb{R}^{N}\times \mathbb{R}^{N}}\frac{|u_{k}(x)-u_{k}(y)|^{p(x,y)}}{|x-y|^{N+sp(x,y)}}\,dx\,dy - \lambda\int_{\Omega}F(x,u_{k})\,dx+o_{k}(1).
\end{split}
\end{equation*}
Then,
\begin{equation}\label{133}
\int_{\Omega}F(x,u_{k})\,dx \geqslant \frac{c_{\mathcal{A}}b_{0}}{\lambda p^{+}}\|u_{k}\|^{p^{-}}_{\mathscr{W}}-\frac{c}{\lambda}+\frac{o_{k}(1)}{\lambda}\to+\infty 
\mbox{ as } k\to+\infty.
\end{equation}
On the other hand, from $(a_{1})$,  $(a_{2})$, and $(\mathcal{K})$,  we also that
\begin{equation*}
\begin{split}
c   \leqslant & \frac{C_{\mathcal{A}}b_1 }{p^{-}}\|u_k\|^{p^{+}}_{\mathscr{W}}- \lambda\int_{\Omega}F(x,u_{k})\,dx+o_{k}(1)  \mbox{ for all  } k \in \mathbb{R}.
\end{split}
\end{equation*} 
Consequently, by  $(\ref{133}),$ we achieve
\begin{equation}\label{14}
\|u_{k}\|^{p^{+}}_{\mathscr{W}}\geqslant \frac{cp^{-}}{C_{\mathcal{A}}b_1}+\frac{\lambda p^{-}}{C_{\mathcal{A}}b_1}\int_{\Omega}F(x,u_{k})\,dx- \frac{p^{-}}{C_{\mathcal{A}}b_1}o_{k}(1)>0,
\end{equation}
for $k$ large enough.

\noindent We claim that $|\Omega_{\star}|=0.$ Indeed, if $|\Omega_{\star}|\neq 0,$ then by $(\ref{8}),$ $(\ref{12}),$ $(\ref{14}),$ 
and Fatou's Lemma, we have
\begin{equation}\label{155}
\begin{split}
+\infty&=\int_{\Omega_{\star}}\liminf_{k\to + \infty}\frac{F(x,u_{k})}{|u_{k}|^{p^{+}}}|\omega_{k}(x)|^{p^{+}}\,dx-\int_{\Omega_{\star}}\limsup_{k\to+\infty}\frac{c_{6}}{\|u_{k}\|^{p^{+}}_{\mathscr{W}}}\,dx\\ &\leqslant \liminf_{k\to+\infty}\frac{\displaystyle{\int_{\Omega}F(x,u_{k})\,dx}}{\frac{\lambda p^{-}}{C_{\mathcal{A}}b_{1}}\displaystyle{\int_{\Omega}F(x,u_{k})\,dx}-o_{k}(1)}\cdot
\end{split}
\end{equation}

\noindent Therefore, from \eqref{14} and \eqref{155}, we obtain that $+\infty\leqslant \frac{C_{\mathcal{A}}b_1}{\lambda p^{-}}$,
which is a contradiction. This proves that $|\Omega_{\star}|=0$ and thus $\omega(x)=0$ a.e. $x\in  \Omega.$

\noindent Now as in \cite{jeanjean}, we define the continuous function $B_{k}:[0,1]\longrightarrow \mathbb{R}$ by $B_{k}(t):=\Psi_{\lambda}(t u_{k}).$   
Since $B_{k}(t):=\Psi_{\lambda}(tu_{k})$ is continuous in $[0,1],$ we can say that for each $k\in\mathbb{N}$ 
there exists $t_{k}\in [0,1]$ such that
\begin{equation}\label{16}
\Psi_{\lambda}(t_{k}u_{k}):=\max_{t\in[0,1]}B_{k}(t).
\end{equation}
(If for $k\in \mathbb{N}$ is not unique we choose the smaller possible value).
Note that  $t_{k}>0$ for all $k\in\mathbb{N}.$ 
Indeed, passing to a subsequence if necessary, we have $\displaystyle{\Psi_\lambda(u_k)\geq \dfrac{c}{2}}$ for all $k\in\mathbb{N}.$ So, 
if $t_{k}=0$ for all $k\in\mathbb{N}$ it follows that
\begin{equation}\label{estrelaa1}
\Psi_{\lambda}(t_{k}u_{k})=\Psi_{\lambda}(0)=0,
\end{equation}
however,
\begin{equation}\label{estrelaa2}
\ 0<\dfrac{c}{2}\leqslant  \Psi_\lambda(u_k)\leqslant 
    \max_{t\in[0,1]}\Psi_{\lambda}(t u_{k})=\Psi_{\lambda}(t_{k}u_{k}).
\end{equation}
Thus, from \eqref{estrelaa1} and \eqref{estrelaa2}, we obtain a contradiction.

\noindent If $t_{k}\in (0,1),$ by \eqref{16}, we infer  that 
  $$ 
      \frac{d}{dt}_{\mid_{t=t_{k}}}\Psi_{\lambda}(tu_{k})=0\,\, \mbox{ for all } \hspace{0.1cm} k\in \mathbb{N}.
  $$
 Moreover, if $t_{k}=1,$  by \eqref{6} we have $\langle \Psi_{\lambda}'(u_{k}), u_{k}\rangle=o_{k}(1).$
 So we always have
\begin{equation*}
\langle \Psi_{\lambda}'(t_{k}u_{k}),t_{k}u_{k}\rangle= t_{k}\frac{d}{dt}_{\mid_{t=t_{k}}}\Psi_{\lambda}(tu_{k})=o_{k}(1).
\end{equation*} 

\noindent Let $(r_{j})_{j \in \mathbb{N}}$ be a positive sequence of real numbers such that $r_{j}>1$ and $\displaystyle{\lim_{j\to+\infty}r_{j}=+\infty.}$ 
Then $\|r_{j}\omega_{k}\|_{\mathscr{W}}=r_{j}>1$ for all $j$ and $k \in \mathbb{N}.$ Fix $j \in \mathbb{N},$ since $\omega_{k}\to 0$  in $L^{\vartheta(\cdot)}(\Omega)$ 
and $\omega_{k}(x)\to 0$ a.e. $x\in \Omega,$ as $k\to +\infty,$  using the condition $(f_{0}),$ there exists a positive constant $c_{7}$ such that
\begin{equation}\label{f11}
\Big|F(x,r_{j}\omega_{k})\Big|\leqslant c_{7}\left(r_{j}|\omega_{k}(x)| +r_{j}^{\vartheta(x)}|\omega_{k}(x)|^{\vartheta(x)}\right)
\end{equation}
and  by  continuity of the function $F,$ we achieve
\begin{equation}\label{f33}
F(x,r_{j}\omega_{k})\to F(x,r_{j}\omega)=0 \mbox{ a.e. }x\in \Omega \mbox{ as } k\to +\infty,
\end{equation}
for each $j\in\mathbb{N}.$
Consequently, from \eqref{f11}, \eqref{f33}, and the Dominated Convergence Theorem, we obtain
\begin{equation}\label{18'}
\lim_{k\to +\infty}\int_{\Omega}F(x, r_{j}\omega_{k})\,dx=0
\end{equation} 
for all $j\in \mathbb{N}.$

\noindent  Since $\|u_{k}\|_{\mathscr{W}}\to +\infty$ as $k\to +\infty,$ we also  have either  $\|u_{k}\|_{\mathscr{W}}>r_{j}$ or $\frac{r_{j}}{\|u_{k}\|_{\mathscr{W}}}\in (0,1)$ for $k$ large enough. Thus, using \eqref{16}, \eqref{18'}, $(a_{2})$, $(a_{3})$,  $(\mathcal{K})$, and Proposition \ref{lw0}, we deduce that
\begin{equation}\label{19}
\begin{split}
\Psi_{\lambda}(t_{k}u_{k})\geqslant \Psi_{\lambda}\left(\frac{r_{j}}{\|u_{k}\|_{\mathscr{W}}}u_{k}\right)=\Psi_{\lambda}(r_{j}\omega_{k})\geqslant \frac{c_{\mathcal{A}}b_0 r_{j}^{p^{-}}}{p^{+}}-\lambda \int_{\Omega}F(x,r_{j}\omega_{k})\,dx 
\end{split}
\end{equation}
$\mbox{ for all } k  \mbox{ large enough }.$

\noindent Therefore, by $\eqref{19}$ letting $k,j\to +\infty$, we conclude
\begin{equation}\label{20}
\limsup_{k\to+\infty}\Psi_{\lambda}(t_{k}u_{k})=+\infty.
\end{equation} 
Now, we affirm that
$\displaystyle{\limsup_{k\to+\infty}\Psi_{\lambda}(t_{k}u_{k})\leqslant \delta},$
for a suitable positive constant $\delta.$
Indeed, from 
 $(a_4),$ $(f_{3}),$ $\eqref{6},$ and for all $k$ large enough, we have
\begin{equation*}
\begin{split}
\Psi_{\lambda}(t_{k}u_{k}) =& \Psi_{\lambda}(t_{k}u_{k})-\frac{1}{p^{+}}\langle\Psi_{\lambda}'((t_{k}u_{k})),t_{k}u_{k} \rangle+o_{k}(1) \\ =& \frac{1}{p^{+}}\mathcal{H}(t_k u_k) + \frac{\lambda}{p^{+}}\int_{\Omega}\mathcal{G}(x, t_{k}u_{k})\,dx+o_{k}(1) \\ \leqslant & \frac{1}{p^{+}}\int_{\mathbb{R}^{N}\times \mathbb{R}^{N}}\mathcal{H}(u_k)K(x,y)\,dx\,dy + \frac{\lambda}{p^{+}}\int_{\Omega}(\mathcal{G}(x, u_{k})+c_{\star})\,dx+o_{k}(1) \\ =& \Psi(u_k)- \frac{1}{p^{+}}\langle\Psi_{\lambda}'(u_{k}),u_{k} \rangle+ \frac{\lambda c_{\star}|\Omega|}{p^{+}}+ o_{k}(1). 
\end{split}
\end{equation*}
Then,
\begin{equation}\label{255}
\Psi_{\lambda}(t_{k}u_{k}) \longrightarrow c
+\dfrac{\lambda \mathsf{c}_{\star}}{p^{+}}|\Omega|
\mbox{ as } k\longrightarrow +\infty.
\end{equation}
Consequently, from  \eqref{20} and \eqref{255} we obtain a contradiction. Therefore, the sequence $(u_{k})_{k\in\mathbb{N}}$ is bounded in $\mathscr{W}.$

\noindent Now, with standard arguments, we prove that any $(C)_{c}$ sequence has a convergent subsequence. Since  $\mathscr{W}$ is a reflexive  Banach space  there exists $u\in \mathscr{W}$ such that, up to a subsequence still denoted by $(u_{k})_{k\in\mathbb{N}}$ we obtain that  $u_{k}\rightharpoonup u $ in $ \mathscr{W}$
and by Lemma \ref{2.11}, we achieve
\begin{equation*}
\begin{array}{c}
u_{k}(x)\to u(x) \mbox{ a.e. } x\in \Omega, \;\;
u_{k}\to u \mbox{ in } L^{\vartheta(\cdot)}(\Omega),\mbox{ and } u_{k}\to u \mbox{ in } L^{\mathsf{m}(\cdot)}(\Omega)\mbox{ as } k \to +\infty.
\end{array} 
\end{equation*}
Hence, using the H\"older's inequality, we have

\begin{equation}\label{377}
\Big|\int_{\Omega}f(x,u_{k})(u_{k}-u)\,dx \Big|
\leqslant C\|1+|u_{k}|^{\vartheta(x)-1}\|_{L^{\vartheta'(\cdot)}(\Omega)}
\|u_{k}-u\|_{L^{\vartheta(\cdot)}(\Omega)}\to 0 \mbox{ as } k\to+\infty.
\end{equation} 
and from \eqref{6}, it follows that 
\begin{equation}\label{606}
\langle\Psi'_{\lambda}(u_{k}),u_{k}-u \rangle\to 0 \mbox{  as } k\to +\infty.
\end{equation}
Thus,  using \eqref{377} and \eqref{606},  we get 
$$\langle\Phi'(u_{k}),u_{k}-u \rangle=\lambda \int_{\Omega}f(x,u_{k})(u_{k}-u)\,dx+\langle\Psi'_{\lambda}(u_{k}),u_{k}-u \rangle\to 0\,\,as\,\,k\to +\infty.$$
Thus, since the operator $\Phi'$ is of type $(S_{+})$ (see Lemma \ref{ll1}), we conclude that $u_{k}\to u$  in $\mathscr{W}.$ 
Therefore, this proves that $\Psi_{\lambda}$ satisfies the $(C)_{c}$ condition on $\mathscr{W}$ and we finish the proof of Lemma.
\end{proof}
\begin{proof}[Proof of Theorem \ref{cc}]
From Lemma \ref{lema 1} and Lemma \ref{cerami}, the Euler Lagrange functional $\Psi_{\lambda}$ satisfies the geometric conditions the Mountain Pass Theorem. Moreover $\Psi_{\lambda}(0)=0$. Therefore, by Theorem \ref{MountainPass}, the functional $\Psi_{\lambda}$ has a critical value $c\geqslant \mathcal{R}>0,$ that is,  exists $u \in \mathscr{W}$ such that  problem \eqref{e1.1} has at least one nontrivial weak solution in $\mathscr{W}.$
\end{proof}

 \subsection{Proof of Theorem \ref{t7}}
  \hfill \break  To prove   Theorem \ref{t7} we use some Lemmas presented below and   Theorem \ref{psm} what is $ '' \mathbb{Z}_{2}- symmetric ''$ version (for even functionals) of the Mountain Pass.

  \begin{lemma}\label{mp1}
  Assume $(a_1)$-$(a_3)$, $(\mathcal{K})$, and $(f_0)$-$(f_2)$ are  fulfilled. Then for each $\lambda >0$, there exist $\mathcal{R}_{1}>0$ and $ r_{1}>0$ such that $  \Psi_{\lambda}(u)\geqslant \mathcal{R}_{1}$ for all $ u \in \mathscr{W} $ with $ \|u \|_{\mathscr{W}} = r_{1}.$
 \end{lemma}
  
   \begin{proof}
   The proof is as in the Lemma \ref{lema 1}.
   \end{proof}
  \begin{lemma}\label{pm2}
Assume $(a_1)$-$(a_3)$, $(\mathcal{K})$, and $(f_1)$ are  fulfilled. For every finite dimensional subspace $\widehat{\mathscr{W}} \subset \mathscr{W}$  there exists $\mathcal{R}_2 =
\mathcal{R}_2(\widehat{\mathscr{W}})$ such that
\begin{equation*}
\Psi_{\lambda}(u)\leqslant 0 \mbox{ for all } u \in \widehat{\mathscr{W}} \setminus B_{\mathcal{R}_2}(\mathscr{W}),
\end{equation*}
where $B_{\mathcal{R}_2}(\mathscr{W}) = \{ u \in\widehat{\mathscr{W}} : \| u \|_{\mathscr{W}} < \mathcal{R}_{2} \}$.
 \end{lemma}
  \begin{proof}
 Consider $\widehat{\mathscr{W}}$ be a finite dimensional  subspace of $ \mathscr{W}$ and let $u \in \widehat{\mathscr{W}}$ with $\|u \|_{ \mathscr{W}}=1$ fixed.   Thus,  for all  $t \geqslant 1$  using   $(a_{1})$,  $(a_{2})$, $(\mathcal{K})$, $(f_1)$, \eqref{1l1}, and   Proposition \ref{lw0}, we get
\begin{equation*}\label{dda1}
\begin{split}
\Psi_{\lambda}(tu)\leqslant\frac{C_{\mathcal{A}}b_1}{p^{-}}t^{p^{+}} -\lambda \mathscr{C} c_{p^{+}}t^{p^{+}} + \lambda c_{\mathscr{C}}|\Omega|.
\end{split}
\end{equation*}
Taking $\Pi(t)=\bigg[  \frac{C_{\mathcal{A}}b_1}{p^{-}}  - \lambda \mathscr{C} c_{p^{+}}\bigg]t^{p^{+}}+ \lambda c_{\mathscr{C}}|\Omega|$, if  $\mathscr{C}$ is  large enough, then $\Pi(t)\to -\infty$ as $t\to+\infty$.

\noindent Therefore, we have
\begin{equation*}
\begin{split}
\sup\,\{ \Psi_{\lambda}(u): u \in  \widehat{\mathscr{W}}, \|u \|_{  \mathscr{W}}=\mathcal{R}_{3}\} = \sup\,\{\Psi_{\lambda}(\mathcal{R}_{3}u): u \in  \widehat{\mathscr{W}}, \|u \|_{ \mathscr{W}}=1\} \to - \infty 
\end{split}
\end{equation*}
$ \mbox{ as } \mathcal{R}_{3} \to +\infty.$

\noindent Hence, there exists $\mathcal{R}_2 > 0$ sufficiently large such that $\Psi_{\lambda}(u) \leqslant 0$ for all $u \in \widehat{\mathscr{W}}$ with $\|u \|_{  \mathscr{W}}=\mathcal{R}_{3}$ and $\mathcal{R}_{3} \geqslant \mathcal{R}_2$.
   \end{proof}

\begin{proof}[Proof of Theorem \ref{t7}]
From Lemma \ref{lema 1}, the Euler Lagrange functional $\Psi_{\lambda}$  satisfies the $(C)_{c}$ condition. Moreover, $\Psi_{\lambda}(0)=0$ and  $\Psi_{\lambda}$ is even functional by conditions $(a_1)$ and $(f_4)$.
 Therefore, using the Lemma \ref{mp1}, Lemma \ref{pm2}, and Theorem \ref{psm} we conclude  the existence of an unbounded sequence of weak solutions to  problem \eqref{e1.1} and this completes the proof. 
\end{proof}

  \subsection{Proof of Theorem \ref{fountain}}
   \hfill \break
To prove the  Theorem \ref{fountain} we verify the  hypotheses of  Theorem \ref{foun}  through the Lemmas presented below.

 Since that $\mathscr{W}$ is a separable and reflexive Banach space, using  (\cite[Chapter 4]{fabian},  or \cite[Section 17]{zhao}) or \cite{schauder},  there exist sequence $(e_l)_{l \in \mathbb{N}}\subset \mathscr{W}$ and $(e^{\star}_{l})_{\mathbb{N}}\subset \mathscr{W}'$ such that
$$\mathscr{W}=\overline{span\{e_{l}:l=1,2,\ldots\}},
\; \mathscr{W}'=\overline{span\{e_{l}^{\star}:l=1,2,\ldots\}}^{\omega^{\star}},$$
and
$$ \langle e_{i}^{\star},e_{l}\rangle
= \left\{\begin{array}{rll}
1 & \hbox{se} & i=l, \\
0 & \hbox{se} & i\neq l.
\end{array}\right.$$

\noindent We denote
$$\mathscr{W}_{l}  =span\{e_{l}\},\; Y_{j}=\bigoplus_{l=1}^{j}\mathscr{W}_{l}=span\{e_{1},\ldots,e_{j}\},\mbox{ and } Z_{j}=\overline{\bigoplus_{l=j}^{\infty}\mathscr{W}_{l}}=\overline{span\{e_{j},e_{j+1}\ldots\}}.$$

 \begin{lemma}\label{beta}
If $\vartheta\in C^{+}(\overline{\Omega}),\,\,p^{+}<\vartheta(x)<p^{\star-}_{s}$ for all $x\in\overline{\Omega}$, denote
$$\beta_{j}:=\sup\{\|u\|_{L^{\vartheta(\cdot)}(\Omega)}:\|u\|_{\mathscr{W}}=1,\,u\in Z_{j}\},$$
then $\displaystyle{\lim_{j\to+\infty}\beta_{j}=0}.$
\end{lemma}
\begin{proof}
It is clear that, $0\leqslant\beta_{j+1}\leqslant\beta_{j}$, thus $\beta_{j}\to\beta\geqslant 0$ as $j\to+\infty . $ Let $u_{j}\in Z_{j}$ satisfy
$$\|u_{j}\|_{\mathscr{W}}=1,\,\,0\leqslant \beta_{j}-\|u_{j}\|_{L^{\vartheta(\cdot)}(\Omega)}<\frac{1}{j} \mbox{ for } j \in \mathbb{N}.$$
Since $\mathscr{W}$ is reflexive there exist a subsequence of $(u_{j})_{j\in \mathbb{N}}$ such that $u_{j}\rightharpoonup u$ as $j \to +\infty$. We claim $u=0.$ Indeed, for all $e_{m}^{\star}\in \{e_{l}^{\star}:l=1,2,\ldots\}$, we obtain $\langle e_{m}^{\star},u_{j}\rangle=0,\,\,j>m.$ So $\langle e_{m}^{\star},u_{j}\rangle\to 0\,\,as\,\,j\to\infty,$ this concludes $\langle e_{m}^{\star}, u \rangle=0$ for all $e_{m}^{\star}\in\{e_{l}^{\star}:l=1,2,\ldots.\}$. Therefore, $u=0$ and so $u_{j}\to 0$ as $j\to+\infty $. Since the embedding from $\mathscr{W}\hookrightarrow L^{\vartheta(x)}(\Omega)$ is compact, then $u_{j}\to 0\,\,$ in $\,\,L^{\vartheta(x)}(\Omega)$ as $j\to+\infty $.  Hence we get $\beta_{j}\to 0$ as $j\to+\infty $.
\end{proof}

\begin{lemma}\label{df}
Assume that $\Xi: \mathscr{W}\to\mathbb{R}$ is weakly-strongly continuous, namely, $u_k \rightharpoonup u$ implies $\Xi(u_k) \to \Xi(u)$, and $\Xi(0)=0.$ Then for each $\tau>0$ and $j\in \mathbb{N}$ there exists 
$$\alpha_{j}:=\sup\{|\Xi(u)|: u\in Z_{j},\,\,\|u\|_{\mathscr{W}}<\tau\}<\infty.$$
Moreover, $\displaystyle{\lim_{j\to +\infty}\alpha_{j}=0.}$ 
 \end{lemma}
 \begin{proof}
It is clear that, $0\leqslant\alpha_{j+1}\leqslant\alpha_{j}$, thus $\alpha_{j}\to\alpha\geqslant 0$ $j\to +\infty $.
 Let $u_{j}\in Z_{j},\,\,\|u\|_{\mathscr{W}}\leqslant \tau$ such that
$$0\leqslant \alpha_{j}-|\Xi(u_{j})|<\frac{1}{j}.$$
Since $\mathscr{W}$ is reflexive
there exist a subsequence of $(u_{j})_{j \in \mathbb{N}}$ such that $u_{j}\rightharpoonup u$. 
Proceeding as in the previous lemma, we obtain $u=0$. The weak- strong continuity of $\Xi$, guarantees $\Xi(u_{j})\to \Xi(0)=0 $. Therefore, $\alpha_{j}\to 0$ as $j\to +\infty $.\end{proof}

\begin{lemma}\label{f1}
Assume   $(a_{1})$-$(a_{3}),$ $(\mathcal{K}),$ $(f_{0})$,   and $(f_{2})$ are satisfied. Moreover, let $\displaystyle{\lambda \in\bigg( 0, \frac{c_{\mathcal{A}}b_0}{p^{+}}\bigg)}$. Then there exist $\rho_{j}>r_{j}>0$ such that:
\begin{itemize}
\item[$(h_2)$]  $b_{j}
:=\inf\{ \Psi_{\lambda}(u): u\in Z_{j},\,\|u\|_{\mathscr{W}}=r_{j}\} 
\to +\infty \mbox{ as } j\to+\infty;$
\item[$(h_3)$] $a_{j}:=\max\{ \Psi_{\lambda}(u) : 
u\in Y_{j},\,\|u\|_{\mathscr{W}}=\rho_{j}\}\leqslant 0.$
\end{itemize}
\end{lemma}
 \begin{proof} $(h_2)$ First, notice that for all $u\in Z_{j}$ with $\|u\|_{\mathscr{W}}>1,$ using $(a_{2})$, $(a_{3})$, $(\mathcal{K})$, and  $(f_{0})$, we infer
\begin{equation}\label{tt1}
\begin{split}
\Psi_{\lambda}(u)\geqslant  &  \frac{c_{\mathcal{A}}b_{0}}{p^{+}}\int_{\mathbb{R}^{N}\times \mathbb{R}^{N}}\frac{|u(x)-u(y)|^{p(x,y)}}{|x-y|^{N+sp(x,y)}}\,dx\,dy -\lambda c_{1}\int_{\Omega} \bigg(|u|+\frac{|u|^{\vartheta(x)}}{\vartheta(x)}\bigg)\,dx \\ \geqslant &  \frac{c_{\mathcal{A}}b_{0}}{p^{+}}\|u\|^{p^{-}}_{\mathscr{W}}-\lambda \frac{ c_{1}\beta_{j}^{\underline{\vartheta}^{+}}}{\underline{\vartheta}^{-}}\|u\|^{\underline{\vartheta}^{+}}_{\mathscr{W}}-\lambda c_{8}\|u\|_{\mathscr{W}},
\end{split}
\end{equation}
 for a constant $c_8>0$ and $\beta_{j}:=\sup\{\|u\|_{L^{\vartheta(x)}(\Omega)}:\|u\|_{ \mathscr{W}}=1,\,u\in Z_{j}\}.$ 

\noindent Now, since $p^{-}\leqslant p^{+}<\underline{\vartheta}^{+},$  by the Lemma \ref{beta},  it is easy see that 
$r_{j}:=(c_{1}\beta_{j}^{\underline{\vartheta}^{+}})^{\frac{1}{p^{-}-\underline{\vartheta}^{+}}}\to +\infty$ as $j\to+\infty.$ Then, for $j$ sufficiently large, 
$u\in Z_{j}$ with $\|u\|_{\mathscr{W}}=r_{j}>1,$ and by $(\ref{tt1}),$ we conclude
\begin{equation*}
\begin{split}
\Psi_{\lambda}(u) \geqslant \left(\frac{c_{\mathcal{A}}b_{0}}{p^{+}}-\lambda \right)r_{j}^{p^{-}}-\lambda c_{8}r_{j}.
\end{split}
\end{equation*}
Therefore, since $p^->1$ and $\displaystyle{\frac{c_{\mathcal{A}}b_{0}}{p^{+}}>\lambda,}$ we obtain that $b_{j}
:=\inf\{ \Psi_{\lambda}(u): u\in Z_{j},\,\|u\|_{\mathscr{W}}=r_{j}\}\to +\infty$ as $j\to +\infty.$
\\
\\
 \noindent $(h_3)$ Note that for all $v \in Y_{j}$ with $\|v\|_{\mathscr{W}}=1,\,$ using $(a_{2}),$ $(a_{3}),$  $(\mathcal{K})$, $(f_0)$ and  \eqref{1l1} for $t>1,$ we have
\begin{equation}\label{a2}
\begin{split}
\Psi_{\lambda}(tv)
 \leqslant t^{p^{+}}\left(\frac{C_{\mathcal{A}}b_{1}}{p^{-}}\|v\|_{\mathscr{W}}^{p^{+}}-\lambda \mathscr{C}\int_{\Omega}|v|^{p^{+}}\,dx \right)+\lambda c_{\mathscr{C}}|\Omega|.
\end{split}
\end{equation}
It is clear that we can choose $\mathscr{C}>0$ large enough such that
$$\frac{C_{\mathcal{A}}b_{1}}{p^{-}}\|v\|^{p^{+}}_{\mathscr{W}}-\lambda \mathscr{C}\int_{\Omega}|v|^{p^{+}}\,dx <0.$$
From \eqref{a2}, it follows that 
$$\lim_{t\to+\infty}\Psi_{\lambda}(tv)=-\infty.$$
Therefore, there exists $t_{0}>r_{j}>1$ large enough such that $\Psi_{\lambda}(t_{0}v)\leqslant 0$ and thus, if set $\rho_{j}=t_{0},$ we conclude that
$$a_{j}:=\max \lbrace \Psi_{\lambda}(u): u\in Y_{j};\,\,\|u\|_{\mathscr{W}}=\rho_{j} \rbrace \leqslant 0.$$

\end{proof}
\begin{proof}[Proof of Theorem \ref{fountain}]
From conditions $(a_{1})$ and $(f_{4})$  the Euler Lagrange functional $\Psi_{\lambda}$ is  even functional and by Lemma \ref{cerami},  $\Psi_{\lambda}$ satisfies   the $(C)_{c}$ condition for every $c>0.$ Then, with that and by Lemma \ref{f1}   all conditions of the  Theorem \ref{foun} are satisfied. Therefore, we obtain a sequence of critical points $(u_{k})_{k\in\mathbb{N}}$ in $\mathscr{W}$ such that $\Psi_{\lambda}(u_{k})\to +\infty$ as $k\to+\infty.$ 
\end{proof}
 \subsection{Proof of Theorem \ref{t5}}
  \hfill \break
  To prove  Theorem \ref{t5} we use the Lemma presented below and we verify hypotheses Theorem \ref{dual}.

\begin{lemma}\label{Cc'} 
Suppose that the hypotheses in Theorem \ref{t5} hold, then $\Psi_{\lambda}$ satisfied the $(C)_{c}^{\star}$ condition.
\end{lemma} 
\begin{proof}
Let $c\in \mathbb{R}$ and the sequence $(u_{k})_{k\in\mathbb{N}}\subset \mathscr{W}$ be such that  $u_{k}\in Y_{k}$ for all $k\in \mathbb{N},$ $\Psi_{\lambda}(u_{k})\to c$ and $\big\|\Psi'_{\lambda}{_{|_{Y_{k_{j}}}}}(u_{k})\big\|_{\mathscr{W}'}(1+\|u_{k}\|_{\mathscr{W}})  \to 0,$ as $k\to +\infty.$
Therefore, we have
$$c=\Psi_{\lambda}(u_{k})+o_{k}(1) \mbox{ and }\langle \Psi'_{\lambda}(u_{k}),u_{k} \rangle=o_{k}(1),$$
where $o_{k}(1)\to 0$ as $k\to+\infty.$
Analogously to the proof of Lemma \ref{cerami}, we can prove that the sequence 
$(u_{k})_{k\in\mathbb{N}}$ is bounded in $\mathscr{W}$. Since $\mathscr{W}$ is reflexive, we can extract a subsequence of $(u_{k})_{k\in\mathbb{N}},$ denoted for $(u_{k_{j}})_{j\in\mathbb{N}},$ and $u \in \mathscr{W} $ such that $u_{k_{j}}\rightharpoonup u$  in $\mathscr{W}$ as $j \to +\infty$.

\noindent On the other hand, as $\mathscr{W}=\overline{\cup_{k}Y_{k}}=\overline{span\{e_{k}:k\geqslant 1\}},$ we can choose $v_{k}\in Y_{k}$ such that $v_{k}\to u$  in $\mathscr{W}$ as $k \to +\infty$.
Hence, we conclude that
\begin{equation*}
\langle \Psi'_{\lambda}(u_{k_{j}}),u_{k_{j}}-u\rangle=\langle \Psi_{\lambda}'(u_{k_{j}}),u_{k_{j}}-v_{k_{j}}\rangle+ \langle \Psi_{\lambda}'(u_{k_{j}}),v_{k_{j}}-u\rangle.
\end{equation*}
Since $\Psi'_{\lambda}{_{|_{Y_{k_{j}}}}}(u_{k_{j}})\to 0$ and  $u_{k_{j}}-v_{k_{j}}\rightharpoonup 0$ in $Y_{k_{j}},$ as $j\to +\infty,$  we achieve
$$\lim_{j\to+\infty}\langle \Psi_{\lambda}'(u_{k_{j}}),u_{k_{j}}-u\rangle=0.$$
Furthermore, using H\"older's inequality, we obtain that
$$\int_{\Omega}f(x,u_{k_{j}})(u_{k_{j}}-u)\,dx \to 0 \mbox{ as  } j\to+\infty.$$
Therefore, 
$$\langle \Phi'(u_{k_{j}}),u_{k_{j}}-u\rangle=\lambda\int_{\Omega}f(x,u_{k_{j}})(u_{k_{j}}-u)\,dx+\langle \Psi_{\lambda}'(u_{k_{j}}),u_{k_{j}}-u\rangle\to 0\,\,as\,\,j\to+\infty.$$
Consequently, since $\Phi'$ is of type $(S_{+}),$ it follows that $u_{k_{j}}\to u$  in $\mathscr{W}$ as $j \to +\infty$.
Then, we conclude that $\Psi_{\lambda}$ satisfies the $(C)_{c}^{\star}$ condition. Thus, we obtain that $\Psi_{\lambda}'(u_{k_{j}})\to  \Psi_{\lambda}'(u)$ as $j\to +\infty.$ 

\noindent Let us prove $\Psi_{\lambda}'(u)=0$. Indeed, taking $\omega_{l}\in Y_{l}$, notice that when $k_{j}\geqslant l,$ we achieve
\begin{eqnarray*}
\langle \Psi_{\lambda}'(u),\omega_{l}\rangle&=&\langle \Psi_{\lambda}'(u)-\Psi_{\lambda}'(u_{k_{j}}),\omega_{l}\rangle+\langle\Psi_{\lambda}'(u_{k_{j}}),\omega_{l}\rangle\\
&=&\langle \Psi_{\lambda}'(u)-\Psi_{\lambda}'(u_{k_{j}}),\omega_{l}\rangle+\langle\Psi'_{\lambda}{_{|_{Y_{k_{j}}}}}(u_{k_{j}}),\omega_{l} \rangle,
\end{eqnarray*}
so, passing the limit on the right side of the equation above as $j\to +\infty,$ we conclude
$$\langle \Psi_{\lambda}'(u),w_{l}\rangle=0 \mbox{ for all }
    \omega_{l}\in Y_{l}.$$
Therefore, $\Psi_{\lambda}'(u)=0$ in $\mathscr{W}'$  and this show that $\Psi_{\lambda}$ satisfies the $(C)_{c}^{\star}$ condition for every $c\in \mathbb{R}.$
 \end{proof}

 \begin{proof}[Proof of Theorem \ref{t5}]
First we observe that from  $(a_{1})$, $(f_{3})$, and Lemma \ref{Cc'} the Euler Lagrange functional $\Psi_{\lambda} $ is even functional and satisfies the $(C)_{c}^{\star}$ condition for all $c\in\mathbb{R}$. 

\noindent Now we will show that the conditions $(g_{1}),\,(g_{2}),\,$ and $(g_{3})$ of the Dual Fountain Theorem are satisfied:

\noindent $(g_{1})$ First we note   that using $(f_0)$ and Young's inequality, we have
\begin{equation}\label{100}
|F(x,t)|\leqslant c_{9}(1+|t|^{\vartheta(x)})
\end{equation}
for a  positive constant $c_{9}$. Moreover,  as $\displaystyle{\lambda<\frac{c_{\mathcal{A}}b_{0} }{p^{+}}},$ we have
$$ \lim_{j\to +\infty}\left( \frac{c_\mathcal{A}b_{0}}{p^{+}}-\lambda\right)(c_{9}\beta_{j}^{\underline{\vartheta}^{+}})^{\frac{p^{-}}{p^{-}-\underline{\vartheta}^{+}}}=+\infty.$$
Then, there exists $j_{0}\in \mathbb{N}$ such that
$$\left( \frac{c_\mathcal{A}b_{0}}{p^{+}}-\lambda\right)(c_{9}\beta_{j}^{\underline{\vartheta}^{+}})^{\frac{p^{-}}{p^{-}-\underline{\vartheta}^{+}}} -\lambda c_{9}|\Omega|\geqslant 0\,\,\mbox{ for all } j\geqslant j_{0}.$$
Taking $\rho_{j}=(c_{9}\beta_{j}^{\underline{\vartheta}^{+}})^{\frac{1}{p^{-}-\underline{\vartheta}^{+}}}$ for $j\geqslant j_{0}$. It is clear that $\rho_{j}>1$ for all $j\in\mathbb{N},\,\,j\geqslant j_{0},$ since $\displaystyle{\lim_{j\to+{\infty}}\rho_{j}=+\infty}$. 
Using same the arguments of Theorem \ref{fountain}, for all $u\in Z_{j},$ with $\|u\|_{\mathscr{W}}=\rho_{j},$ and using  $(a_{2}),$ $(a_{3}),$ $(\mathcal{K})$, and \eqref{100}, we conclude that
\begin{equation*}
\begin{split}
\Psi_{\lambda}(u)\geqslant \left(\frac{c_{\mathcal{A}}b_{0}}{p^{+}}-\lambda \right)(c_{9}\beta_{j}^{\underline{\vartheta}^{+}})^{\frac{p^{-}}{p^{-}-\underline{\vartheta}^{+}}}-\lambda c_{9}|\Omega|
 \geqslant 0.
\end{split}
\end{equation*}
So, we obtain
$$a_{j}:=\inf\{\Psi_{\lambda}(u): u\in Z_{j},\,\,\|u\|_{\mathscr{W}}=\rho_{j}\}\geqslant 0.$$
\\
\noindent $(g_{2})$ Since $Y_{j}$ is finite dimensional all the norms are equivalent, there exists a constant 
$c_{10}>0$ such that $\|u\|_{L^{p^{+}}(\Omega)}\geqslant c_{10}\|u\|_{\mathscr{W}}$ for all $u\in Y_{j}$.  Then, from $(a_{1}),$ $(a_{2}),$  $(\mathcal{K})$, $(f_1)$, \eqref{1l1}, and Proposition \ref{lw0},  we infer
$$\Psi_{\lambda}(u)\leqslant \frac{C_{\mathcal{A}}b_1}{p^{-}}\|u\|^{p^{+}}_{\mathscr{W}}-\lambda c_{10} \mathscr{C}^{p^{+}}\|u\|^{p^{+}}_{\mathscr{W}}+\lambda
c_{\mathscr{C}}|\Omega| \mbox{ for all } u\in Y_{
j} \mbox{ with }\|u\|_{\mathscr{W}}\geqslant 1 .$$
Let $\displaystyle{\mathscr{B}(t)=\frac{C_{\mathcal{A}}b_1}{p^{-}}t^{p^{+}}-\lambda c_{10} \mathscr{C}^{p^{+} }t^{p^{+}}+\lambda c_{\mathscr{C}}|\Omega|}.$
However,  we can choose $\mathscr{C}>0$ large enough, such that
$$\lim_{t\to+\infty}\mathscr{B}(t)=-\infty.$$
Therefore, there exists $\overline{t}\in (1,+\infty)$ such that
$$\displaystyle{\mathscr{B}(t)<0 \mbox{ for all } t\in [\, \overline{t},+\infty )}.$$
Consequently, $\Psi_{\lambda}(u)<0$ for all $u\in Y_{j}$ with $\|u\|_{\mathscr{W}}=\overline{t}.$
Then, choosing $r_{j}=\overline{t}$ for all $j\in \mathbb{N}$, we have
$$b_{j}:=\max\{\Psi_{\lambda}(u): u\in Y_{j},\,\,\|u\|_{\mathscr{W}}=r_{j}\}<0.$$
We observe that we can change $j_{0}$ on other more large, if necessary, so that $\rho_{j}>r_{j}>0$ for all $j\geqslant j_{0}.$
\\
\noindent $(g_{3})$ Since $Y_{j}\cap Z_{j}\neq \emptyset$ and $0<r_{j}<\rho_{j},$ so, we have
$$d_{j}:=\inf\{\Psi_{\lambda}(u):u\in Z_{j},\,\|u\|_{\mathscr{W}}\leqslant \rho_{j}\} \leqslant b_{j}:=\max\{\Psi_{\lambda}(u):u\in Y_{j},\,\|u\|_{\mathscr{W}}=r_{j}\}<0. $$
 From $(f_{0})$, we obtain $|F(x,u)|\leqslant c_{11}(|t|+|t|^{\vartheta(x)})$,  for a constant positive $c_{11}$, and for all $(x,t) \in \Omega \times\mathbb{R}.$
 Consider, $\Sigma_{1}:\mathscr{W}\to\mathbb{R}$ and 
$\Sigma_{2}:\mathscr{W}\to\mathbb{R}$ defined by
\begin{equation}\label{fi}
\Sigma_{1}(u)=\int_{\Omega}\lambda c_{11}|u|^{\vartheta(x)}\,dx
\mbox{ and }\Sigma_{2}(u)=\int_{\Omega}\lambda c_{11}|u|\,dx.
\end{equation}
We have $\Sigma_{i}(0)=0,$ $i=1,2,$ and they are weakly-strongly continuous. Let us denote
$$\eta_{j}=\sup\{|  \Sigma_{1}(u)|: u\in Z_{j},\,\,\|u\|_{\mathscr{W}}\leqslant 1   \}, \,\,\,\,\xi_{j}=\sup\{|\Sigma_{2}(u)|: u\in Z_{j},\,\,\|u\|_{\mathscr{W}}\leqslant 1  \}.$$
Thus, since the embedding $\mathscr{W}\hookrightarrow L^{\vartheta(\cdot)}(\Omega)$ is compact it follows that by   Lemma \ref{df} that
$\displaystyle{\lim_{j\to +\infty}\eta_{j}=\lim_{j\to+\infty}\xi_{j}=0}.$

\noindent Now, consider $v\in Z_{j}$ with $\|v\|_{\mathscr{W}}=1$ and $0<t<\rho_{j}$. Then, from $(a_{1})$, $(a_{2})$, $(a_{3})$, $(\mathcal{K})$, and $(\ref{fi}),$ we obtain 
\begin{equation*}
\begin{split}
\Psi_{\lambda}(tv) \geqslant-\lambda\int_{\Omega}F(x,tv)\,dx \geqslant  -\Sigma_{1}(tv)-\Sigma_{2}(tv),
\end{split}
\end{equation*}
and since
$$\Sigma_{1}(tv)\leqslant t^{\underline{\vartheta}^{+}}\Sigma_{1}(v)  \mbox{ and }\Sigma_{2}(tv)=t\Sigma_{2}(v),$$
we achieve
$$\Psi_{\lambda}(tv)\geqslant -\rho_{j}^{\underline{\vartheta}^{+}}\Sigma_{1}(v)-\rho_{j}\Sigma_{2}(v)\geqslant -\rho_{j}^{\underline{\vartheta}^{+}}\eta_{j}-\rho_{j}\xi_{j}$$
 for all $t\in (0,\rho_{j})$ and $v\in Z_{j}$ with $\|v\|_{\mathscr{W}}=1.$
 Thus, $d_{j}\geqslant -\rho_{j}^{\underline{\vartheta}^{+}}\eta_{j}-\rho_{j}\xi_{j}$ and as $d_{j}<0$ for all $j\geqslant j_{0}$, we conclude that   $\displaystyle{\lim_{j\to+\infty}d_j=0}.$
  
\noindent  Therefore, the conditions of   Theorem \ref{dual} are satisfied. Consequently, there exists a sequence $(u_k)_{k\in\mathbb{N}}\subset \mathscr{W}$  of weak solutions  of problem  such that $\Psi_{\lambda}(u_k)< 0$ and  $\Psi_{\lambda}(u_k)\to 0$ as $k \to +\infty$  for $\lambda \in \displaystyle{\bigg(0, \frac{c_{\mathcal{A}b_0}}{p^{+}}\bigg)}$.
\end{proof}

 \subsection{Proof of Theorem \ref{t6}}
  \hfill \break
To prove  Theorem \ref{cc} we need of  Lemmas  below and Theorem \ref{MountainPass}.
 
 \begin{lemma}\label{2}
Suppose $(a_1)$-$(a_3)$ and $(f_1)$ holds. There is $v \in\mathscr{W}\setminus \{0\},$ such that  $\displaystyle{\lim_{t\to +\infty}\Psi_{\lambda}(tv)=-\infty}.$
\end{lemma}
\begin{proof}
 The proof is as in the Lemma \ref{lema 1}-(i).
\end{proof} 
 \begin{lemma}\label{1}
Suppose $(a_1)$-$(a_3)$, $(\mathcal{K})$,  $(f_{0})$,   $f(x,0)=0,$ and $f(x,t)\geqslant 0
\mbox{ a.e. }x\in \Omega$ and for all $t\geqslant 0$ holds. Then, there exist $\overline{\lambda}>0$, positive constants $\mathsf{C}_{\lambda}$ and 
$\rho_{\lambda}$ for $\lambda \in (0, \overline{\lambda})$ such that $\displaystyle{\lim_{\lambda\to 0^{+}}\mathsf{C}_{\lambda}=+\infty}$ and $\Psi_{\lambda}(u)>\mathsf{C}_{\lambda}>0$ whenever $\|u\|_{\mathscr{W}}=\rho_{\lambda}.$
\end{lemma}
 \begin{proof}
We consider $u\in \mathscr{W}$ with $\|u\|_{\mathscr{W}}>1$. Then, using  \eqref{100}, Young's inequality, and Lemma \ref{lw0}, we conclude that
\begin{equation}\label{410}
\begin{split}
\Psi_{\lambda}(u)\geqslant&\frac{c_{\mathcal{A}}b_{0}}{p^{+}}\|u\|^{p^{-}}-\lambda c_{13}\|u\|_{\mathscr{W}}^{\underline{\vartheta}^{+}}-\lambda c_{9}|\Omega|
\end{split}
\end{equation}
for constant positive $c_{13}$.

\noindent Let  $\gamma\in \left(0,\frac{1}{\underline{\vartheta}^{+}-p^{-}}\right)$ and $u\in \mathscr{W}$ such that $\|u\|_{\mathscr{W}}=\lambda^{-\gamma}.$ We define 
$\rho_{\lambda}:=\lambda^{-\gamma}$ and we observe that $\rho_{\lambda}>1$ for $\lambda$ small enough. Hence, from \eqref{410}, we conclude
$$\Psi_{\lambda}(u)\geqslant \frac{c_{\mathcal{A}}b_0}{p^{+}}\lambda^{-\gamma p^{-}}-c_{13}\lambda^{1-\gamma \underline{\vartheta}^{+}}-\lambda c_{9}|\Omega|.$$
Since $\gamma<\frac{1}{\underline{\vartheta}^{+}-p^{-}}$ it follows that  $-\gamma p^{-}<1-\gamma \underline{\vartheta}^{+}.$ Thus $\mathsf{C}_{\lambda}:=\frac{c_\mathcal{A}b_{0}}{p^{+}}\lambda^{-\gamma p^{-}}-c_{13}\lambda^{1-\gamma\underline{\vartheta}^{+}}-\lambda c_{9}|\Omega|\to +\infty$ as $\lambda\to 0^{+}.$ Hence, there exists $\overline{\lambda}>0$ small enough such that $\mathsf{C}_{\lambda}>0$ for all $\lambda\in(0,\overline{\lambda}).$
Then we obtain that 
$$\Psi_{\lambda}(u)\geqslant \mathsf{C}_{\lambda}>0=\Psi_{\lambda}(0)$$
for all $u\in\mathscr{W}$ with $\|u\|_{\mathscr{W}}=\rho_{\lambda}=\lambda^{-\gamma}$ and $\lambda\in (0, \overline{\lambda}).$
 Therefore, $\mathsf{C}_{\lambda}$ verifies the assertion of the Lemma.
 \end{proof}

  \begin{proof} [Proof of Theorem \ref{t6}]
From Lemma \ref{cerami} the functional $\Psi_{\lambda}$ satisfy the $(C)_{c}$ condition. Moreover, $\Psi_{\lambda}(0)=0$.  Then, by Lemma \ref{2}, Lemma \ref{1}, and Theorem \ref{MountainPass} we obtain that there are constant $\overline{\lambda}$ and a nontrivial critical point $u_{\lambda}$ for $\Psi_{\lambda}$ with $\lambda \in (0, \overline{\lambda})$ such that
\begin{equation*}
c=\Psi_{\lambda}(u_{\lambda})\geqslant \mathsf{C}_{\lambda}.
\end{equation*}
Then, from $(a_{1}),$ $(a_{2})$, $(\mathcal{K})$, $ (f_{0})$, \eqref{100}, Proposition \ref{lw0},  and Lemma \ref{2.11}, we achieve
\begin{equation}\label{lan}
\begin{split}
\mathsf{C}_{\lambda} \leqslant \Psi_{\lambda}(u_{\lambda})\leqslant \frac{C_{\mathcal{A}}b_1}{p^{-}}\max\{\|u_{\lambda}\|^{p^{+}}_{\mathscr{W}},\|u_{\lambda}\|^{p^{-}}_{\mathscr{W}}\}+\lambda c_{13}\max\{\|u_{\lambda}\|^{\underline{\vartheta}^{+}}_{\mathscr{W}},\|u_{\lambda}\|^{\underline{\vartheta}^{-}}_{\mathscr{W}}\}+\lambda c_9|\Omega|.
\end{split}
\end{equation}

\noindent Hence, taking $\lambda\to 0^{+}$ in (\ref{lan}) as $\mathsf{C}_{\lambda}\to +\infty,$ we get
$$\lim_{\lambda\to 0^{+}}\|u_{\lambda}\|_{\mathscr{W}}=+\infty.$$ 
Therefore, we conclude the proof of Theorem \ref{t6}.
\end{proof}
 
 \subsection{Proof of Theorem \ref{genus}}
  \hfill \break
  To proof  Theorem \ref{genus} it is enough to verify that $\Psi_{\lambda}$ satisfies the hypotheses of Theorem \ref{kajikiya}. 
  \begin{lemma}\label{PSt6}
Suppose $(a_1)$- $(a_3)$,  $(\mathcal{K})$, and $(f_5)$ holds. Then the functional $\Psi_{\lambda}$ is bounded from below and satisfies the $(PS)$ condition.

\end{lemma}
 \begin{proof}
Let $u \in \mathscr{W}$ and suppose  $\|u\|_{\mathscr{W}}>1.$  Then, from $(a_{2}),$ $(a_{3})$, $(\mathcal{K}),$ $(f_{5}),$  Lemma \ref{2.11} and Lemma\ref{ll1}  for each $\lambda>0,$ we infer  that
\begin{equation*}
\begin{split}
\Psi_{\lambda}(u) \geqslant&  \frac{c_{\mathcal{A}}b_0}{p^{-}}\|u\|^{p^{-}}_{\mathscr{W}}-\frac{\lambda c_{14}}{\underline{\mathfrak{m}}^{-}}\|u\|^{\underline{\mathfrak{m}}^{+}}_{\mathscr{W}}.
\end{split}
\end{equation*}
 Since $\underline{\mathfrak{m}}^{+}<p^{-},$ follows that $\Psi_{\lambda}$ is coercive. Therefore, $\Psi_{\lambda}$ is bounded from below.

\noindent Now, let $(u_{k})_{k\in \mathbb{N}}$ be a sequence in $\mathscr{W}$ such that
$$\Psi_{\lambda}(u_{k})\to c \mbox{ and }\Psi'_{\lambda}(u_{k})\to 0 \mbox{ in } \mathscr{W}' \mbox{ as  }  k \to +\infty.$$
 By coercivity of  $\Psi_{\lambda}$, we have $(u_{k})_{k\in \mathbb{N}}$ is bounded in  $\mathscr{W}.$ Hence up to a subsequence, still denoted by $(u_{k})_{k \in \mathbb{N}},$ we have $u_{0}\in \mathscr{W}$ 
 such that $u_{k}\rightharpoonup u_{0}$  in $ \mathscr{W}$ and from Lemma \ref{lw0} we infer  that  $u_{k}\to u_{0}$ in $ L^{\mathfrak{q}(\cdot)}(\Omega)$ and $u_{k}(x)\to u_{0}(x)$  a.e. $x\in  \Omega$, as $k\to +\infty$.
 
 Since $\Psi'_{\lambda}(u_{k})\to 0$ in $\mathscr{W}$ as $k\to+\infty$  and the sequence $(u_k)_{k \in \mathbb{N}}$ is bounded in $\mathscr{W}$, we have  that
\begin{equation}\label{elma1}
\langle \Psi'_{\lambda}(u_{k}),u_{k}-u_{0}\rangle\to 0\,\,as\,\,k\to+\infty.
\end{equation}
 Now, using  $(f_{5})$ and H\"older's inequality, we obtain that
\begin{equation*}
\begin{split}
  \Big| \int_{\Omega}f(x,u_{k})(u_{k}-u_{0})\,dx \Big| 
  \leqslant&  
  C_{2}\Big\||u_{k}|^{\mathfrak{m}(x)-1}\Big\|_{L^{\frac{\mathfrak{m}(\cdot)}{\mathfrak{m}(\cdot)-1}}(\Omega)}\|u_{k}-u_{0}\|_{L^{\mathfrak{m}(\cdot)}(\Omega)}.
  \end{split}
\end{equation*}  
     Thus, taking into account that $u_{k}\to u_{0}$ in $L^{\mathfrak{m}(\cdot)}(\Omega)$ as $,k\to+\infty$, we achieve
     \begin{equation}\label{elma2}
     \int_{\Omega}f(x,u_{k})(u_{k}-u_{0})\,dx\to 0\,\,as\,\,k\to+\infty.
     \end{equation}
Therefore, from \eqref{elma1} and \eqref{elma2}, we conclude that 
\begin{equation*}
\begin{split}
     \langle \Phi'(u_{k}),u_{k}-u_{0}\rangle =&  \langle \Psi'_{\lambda}(u_{k}),u_{k}-u_{0}\rangle     -\lambda\int_{\Omega}f(x,u_{k})(u_{k}-u_{0})\,dx
       \longrightarrow  0 \mbox{ as } k\to +\infty.
       \end{split}
       \end{equation*}
Since $\Phi'$ is of type $(S_{+})$ (see Lemma \ref{ll1}), we obtain $u_{k}\to u_{0}$  in $\mathscr{W}$ as $k\to+\infty,$ concluding the proof of the Palais-Smale condition.       
\end{proof}

  \begin{proof}[Proof of Theorem \ref{genus}]
 From $(a_{1})$ and  $(f_{4})$  the Euler Lagrange functional $\Psi_{\lambda}$  is even functional. Furthermore, by  Lemma \ref{PSt6}   $\Psi_{\lambda}$  is  bounded from below and satisfies the Palais-Smale condition. Hence the item $(I1)$ of  Theorem \ref{kajikiya} is verified.  Let us show that $\Psi_{\lambda}$ satisfies $(I2)$.  Since $\mathscr{W}$ is a reflexive and separable Banach space, for each $j\in \mathbb{N},$ consider a $j$-dimensional linear subspace $\mathscr{W}_{j}\subset C_{0}^{\infty}(\Omega)$ of $\mathscr{W}.$
We define
$$\mathbb{S}_{\mathcal{R}_{4}}^{j}=\{u\in \mathscr{W}_{j}:\|u\|_{\mathscr{W}}=\mathcal{R}_{4}\}$$
where $\mathcal{R}_{4}>0$ will be determined later on.
Since $\mathscr{W}_{j}$ and $\mathbb{R}^{j}$ are isomorphic and  $\mathbb{S}_{\mathcal{R}_{4}}^{j}$ 
is homeomorphic to the $(j-1)$-dimensional sphere $\mathbb{S}^{j-1}$ by an odd mapping, it has genus $j$, i.e., $\gamma(\mathbb{S}_{\mathcal{R}_{4}}^{j})=j$.

\noindent Now, from $(a_{1})$, $(a_{2})$, $(\mathcal{K})$, and $(f_{5})$, we obtain
\begin{equation}\label{rafe}
\begin{split}
   \Psi_{\lambda}(u)  \leqslant & \frac{C_{\mathcal{A}}b_{1}}{p^{-}}\int_{\mathbb{R}^{N}\times \mathbb{R}^{N}}\frac{|u(x)-u(y)|^{p(x,y)}}{|x-y|^{N+sp(x,y)}}\,dx\,dy-\lambda C_{0}\int_{\Omega}|u|^{\mathfrak{m}(x)}\,dx.
   \end{split}
\end{equation}
By Proposition \ref{lw0}, if $\|u\|_{\mathscr{W}}<1,$ we have 
$\rho_{\mathscr{W}}(u)\leqslant \|u\|^{p^{-}}_{\mathscr{W}},$
and $\rho_{\mathfrak{m}(\cdot)}(u)\geqslant \|u\|_{L^{\mathfrak{m}(\cdot)}(\Omega)}^{\underline{\mathfrak{m}}^{+}}$ for every $u\in \mathscr{W}.$ Moreover, since $\mathscr{W}_{j}$ is a finite dimensional space, any norm in $\mathscr{W}_{j}$ is equivalent to each other. Thus, there exist a constant $C(j)>0$ such that 
$\displaystyle{C(j)\|u\|^{\underline{\mathfrak{m}}^{+}}_{\mathscr{W}}\leqslant \int_{\Omega}|u|^{\mathfrak{m}(x)}\,dx}$ for every $u\in \mathscr{W}_{j}$. Consequently, by \eqref{rafe}, we get
\begin{equation*}
\begin{split}
  \Psi_{\lambda}(u)      \leqslant  \|u\|^{\underline{\mathfrak{m}}^{+}}_{\mathscr{W}}\bigg( \frac{C_{\mathcal{A}}b_{1}}{p^{-}}\|u\|^{p^{-}-\underline{\mathfrak{m}}^{+}}_{\mathscr{W}}-\lambda C(j)C_{0}\bigg),
\end{split}
\end{equation*}
for every $u\in \mathscr{W}_{j}$ with $\|u\|_{\mathscr{W}}<1.$ Let $\mathcal{R}_{5}\in(0,1)$ such that
  $$ 
      \frac{C_{\mathcal{A}}b_{1}}{p^{-}}\mathcal{R}_{5}^{p^{-}-\underline{\mathfrak{m}}^{+}}<\lambda C(j)C_{0}.
  $$
Thus, for all $0<\mathcal{R}_{4}<\mathcal{R}_{5}$ and $u\in \mathbb{S}_{\mathcal{R}_{4}}^{j}$, we have that
 \begin{equation*}
 \begin{split}
   \Psi_{\lambda}(u) \leqslant  \mathcal{R}_{4}^{\underline{\mathfrak{m}}^{+}}\bigg(\frac{C_{\mathcal{A}}b_{1}}{p^{-}}\mathcal{R}_{4}^{p^{-}
   -\underline{\mathfrak{m}}^{+}}-\lambda C(j)C_{0} \bigg)\leqslant \mathcal{R}_{5}^{\underline{\mathfrak{m}}^{+}}\bigg(\frac{C_{\mathcal{A}}b_{1}}{p^{-}}\mathcal{R}_{5}^{p^{-}-\underline{\mathfrak{m}}^{+}}-C(j)C_{0}\bigg)<  0=\Psi_{\lambda}(0).
   \end{split}
\end{equation*}
Therefore, we conclude that
\begin{equation*}\label{ss}
\sup_{u \in \mathbb{S}_{\mathcal{R}_{4}}^{j}} \Psi_{\lambda}(u)<0=\Psi_{\lambda}(0).
\end{equation*}
Thus the hypotheses of Theorem \ref{kajikiya}  are satisfied and we conclude that  there exists a sequence nontrivial weak solutions $u_j$ in $\mathscr{W}$  such that
$$\Psi_{\lambda}(u_j)\leqslant 0, \hspace{0.2cm} \Psi_{\lambda}'(u_j)=0  \mbox{ and } \|u_j\|_{\mathscr{W}}\to 0 \mbox{ as } j\to +\infty.$$

 \end{proof}
  \section{Appendix}\label{apendice}
   \hfill \break
This section presents several results which are used in our paper.

\begin{proof}[Proof of Lemma \ref{lw1}]
We consider the set $ \mathbb{M}=\{ u \in \mathscr{W}; \| u \|_{L^{p(\cdot)}(\Omega)}=1\}$. Then to prove this Lemma  it suffices to prove that
\begin{equation*}
\inf_{u\in \mathbb{M}}[u]^{s, p(\cdot, \cdot)}_{\mathbb{R}^{N}}= \zeta_1>0.
\end{equation*}
Initially, we observe that $\zeta_1 \geqslant 0$ and we prove $\zeta_1$ is attained in $\mathbb{M}$. Let $(u_{k})_{k\in \mathbb{N}} \subset \mathbb{M}$ be a minimizing sequence, that is, $[u_{k}]^{s, p(\cdot, \cdot)}_{\mathbb{R}^{N}}\rightarrow \zeta_1$ as $k \to +\infty$. This implies that $(u_{k})_{k\in \mathbb{N}}$ is bounded in  $ \mathscr{W} $ and $L^{p(\cdot)}(\Omega)$, therefore in $W^{s,p(\cdot, \cdot)}(\Omega)$. Consequently  up to a
subsequence $u_k \rightharpoonup u_0$ in $W^{s,p(\cdot, \cdot)}(\Omega)$ as $k\to +\infty$. Thus, from Corollary \ref{3.4a}, it follows that $u_k \to u_0$ in $L^{p(\cdot)}(\Omega)$ as $k\to +\infty$. We extend $u_0$ to $\mathbb{R}^{N}$ by setting $u_0(x) = 0$ on $x \in \mathbb{R}^{N}\setminus \Omega$. This implies $u_k(x) \to u_0(x)$ a.e. $x \in \mathbb{R}^{N}$ as $k \to +\infty$. Hence by using Fatou’s Lemma, we have
\begin{equation*}
\int_{\mathbb{R}^{N}\times \mathbb{R}^{N}} \frac{|u_0(x)-u_0(y)|^{p(x,y)}}{|x-y|^{N+sp(x,y)}}\,dx\,dy \leqslant \liminf_{k\to +\infty}\int_{\mathbb{R}^{N}\times \mathbb{R}^{N}} \frac{|u_k(x)-u_k(y)|^{p(x,y)}}{|x-y|^{N+sp(x,y)}}\,dx\,dy
\end{equation*}
which implies that
$$[u_{0}]^{s, p(\cdot, \cdot)}_{\mathbb{R}^{N}} \leqslant \liminf_{k\to +\infty}[u_{k}]^{s, p(\cdot, \cdot)}_{\mathbb{R}^{N}}=\zeta_1,$$
and thus $u_0 \in \mathscr{W}$. Moreover,  $\| u_0 \|_{L^{p(\cdot)}(\Omega) }=1$ and then $u_0 \in \mathbb{M} $. In particular, $u_{0}\neq 0$ and $[u_{0}]^{s, p(\cdot, \cdot)}_{\mathbb{R}^{N}}= \zeta_1>0$.
\end{proof}

 \begin{proof}[Proof of Lemma \ref{2.11}]
First we observe that by Lemma \ref{lw1},  for all $ u \in \mathscr{W}$, we get
\begin{equation}\label{liim}
\begin{split}
\|u\|_{W^{s,p(\cdot,\cdot)}(\Omega)}  \leqslant \|u\|_{L^{p(\cdot)}(\Omega)} + \|u\|_{\mathscr{W}}  \leqslant \Bigg(  \frac{1}{\zeta_1} +1 \Bigg) \|u\|_{\mathscr{W}},
\end{split}
\end{equation}
that is,  $ \mathscr{W}$ is continuously embedded in $ W^{s,p(\cdot,\cdot)}(\Omega)$, and by Corollary \ref{3.4a}   we conclude that  $ \mathscr{W} $ is continuously embedded in $ L^{r(\cdot)}(\Omega)$. To prove that the embedding above  is compact  we consider  $(u_{k})_{k\in \mathbb{N}}$ a bounded sequence in $ \mathscr{W}$. Using \eqref{liim}, follows that $(u_{k})_{k\in \mathbb{N}}$ is bounded in $  W^{s,p(\cdot, \cdot)}(\Omega) $. Hence by Corollary \ref{3.4a}, we infer that there  exists $u_0 \in {L^{r(\cdot)}(\Omega)} $  such that up to a subsequence $u_k\to u_0$  in ${L^{r(\cdot)}(\Omega)} $ as $k \to +\infty$.   Since that $u_k = 0$ a.e. in $\mathbb{R}^{N}\setminus\Omega$ for all $k \in \mathbb{N}$,  so we define $u_0 = 0$ a.e. in  $\mathbb{R}^{N}\setminus\Omega$ and obtain that the convergence occurs in $L^{r(\cdot)}(\Omega)$. This completes the  proof this Lemma.

\end{proof}

\begin{proof}[Proof of Lemma \ref{ll1}]
  $(i)$  Using standard arguments proof this item.

\noindent $(ii)$  From $(i)$ the functional $\Phi$ is of   class  $C^{1}(\mathscr{W}, \mathbb{R})$, and by  hypothesis $(a_{1})$,  the functional $\Phi'$ is   monotone. Thus, by  \cite[Lemma 15.4]{kavian}  we conclude that
$\langle \Phi'(u), u_k-u\rangle +\Phi (u) \leqslant \Phi (u_k) $ for all $k \in \mathbb{N}$.

\noindent Thus, since  $u_k \rightharpoonup u$ in $\mathscr{W}$, as $k\to +\infty$ we obtain $\displaystyle{   \Phi (u) \leqslant \liminf_{k\to +\infty}\Phi (u_k)}$. That is, the functional $\Phi$ is weakly lower semicontinuous.
\\
\\
\noindent $(iii)$ Let $(u_{k})_{k\in \mathbb{N}}$ be a sequence in  $\mathscr{W}$ as in the statement. We have that by $(i)$, $\Phi'$ is a continuous  functional. Therefore,
 \begin{equation}\label{ss0}
\displaystyle{\lim_{k\to +\infty}\langle \Phi'(u), u_k-u \rangle = 0}.
\end{equation}
Now, we observe that
\begin{equation}\label{ss1}
\langle \Phi'(u_k)- \Phi'(u), u_k-u \rangle=  \langle \Phi'(u_k), u_k-u \rangle -\langle  \Phi'(u), u_k-u \rangle \mbox{ for all } k \in \mathbb{R}.
\end{equation}
Thus by \eqref{inffo}, \eqref{ss0}, and \eqref{ss1}, we infer
\begin{equation}\label{sa}
\limsup_{k\to +\infty}\langle \Phi'(u_k)- \Phi'(u), u_k-u \rangle \leqslant 0.
\end{equation}
Furthermore, since $\Phi$ is  strictly convex by hypothesis $(a_1)$,   $\Phi'$ is monotone (see \cite[Lemma 15.4]{kavian}), we have
\begin{equation}\label{ss2}
\langle \Phi'(u_k)- \Phi'(u), u_k-u \rangle \geqslant 0 \hspace{0.2cm }\mbox{ for all}  \hspace{0.1cm }k \in \mathbb{N}.
\end{equation}
Therefore, by \eqref{sa} and \eqref{ss2}, we infer that
\begin{equation}\label{sb}
\lim_{k\to +\infty}\langle \Phi'(u_k)- \Phi'(u), u_k-u \rangle = 0.
\end{equation}
Consequently, by  \eqref{ss0}, \eqref{ss1}, and \eqref{sb}, we conclude
\begin{equation}\label{sc}
\lim_{k\to +\infty}\langle \Phi'(u_k), u_k-u \rangle = 0.
\end{equation}
Since $\Phi$ is   strictly convex,  we get
 \begin{equation}\label{sd}
  \Phi(u)+ \langle \Phi'(u_k), u_k-u \rangle \geqslant  \Phi(u_k)\hspace{0.2cm }\mbox{for all} \hspace{0.1cm } k \in \mathbb{N}.
\end{equation}
Thus, by \eqref{sc} and  \eqref{sd}, we have
\begin{equation}\label{se}
   \Phi(u) \geqslant\lim_{k \to +\infty}\Phi(u_k). 
\end{equation}
Since  $\Phi$ is  weakly lower semicontinuous (see $(ii)$) and by \eqref{se}, we conclude that
\begin{equation}\label{ss4}
  \Phi(u) =\lim_{k \to +\infty}\Phi(u_k).
\end{equation}
On the other hand, by \eqref{sb} the sequence $( \mathcal{U}_{k}(x,y) )_{k\in \mathbb{N}}$  converge to $0$   in $L^{1}(\mathbb{R}^{N}\times\mathbb{R}^{N})$ as $k \to +\infty$, where
\begin{equation*}
\mathcal{U}_{k}(x,y):= [\mathcal{A}(u_k(x)-u_k(y))- \mathcal{A}(u(x)-u(y))[(u_k(x)-u_k(y))-(u(x)-u(y)]K(x,y)\geqslant 0.
\end{equation*} 
Hence there exists  a subsequence $(u_{k_{j}})_{j\in \mathbb{N}}$ of  $(u_{k})_{k\in \mathbb{N}}$ such that     
\begin{equation}\label{sg}
\mathcal{U}_{k_j}(x,y)\to 0 \mbox{ a.e.  } (x,y)\in \mathbb{R}^{N}\times \mathbb{R}^{N} \mbox{ as } j \to +\infty. 
\end{equation}
We denoted $\mu_{j}(x,y)=u_{k_j}(x)-u_{k_j}(y)$ and $\mu(x,y)=u(x)-u(y)$  for all $(x,y)\in \mathbb{R}^{N}\times \mathbb{R}^{N}$.
\\
\textbf{Claim a.} If $\mathcal{U}_{k_j}(x,y)\to 0 \mbox{ a.e.  } (x,y)\in \mathbb{R}^{N}\times \mathbb{R}^{N}$, 
then  $\mu_{j}(x,y)\to\mu(x,y)$  as $j\to +\infty$ for almost all $(x,y)\in \mathbb{R}^{N}\times \mathbb{R}^{N}$. \\ Indeed,  fixed $ (x,y)\in \mathbb{R}^{N}\times\mathbb{R}^{N}$ with $x\neq y$   we suppose  by contradiction  that the  sequence $(\mu_{j}(x,y))_{j\in \mathbb{N}}$ is unbounded for $(x,y)\in \mathbb{R}^{N}\times \mathbb{R}^{N}$ fixed. Using \eqref{sg} we get $\mathcal{U}_{k_j}(x,y) \to 0$ as $j\to +\infty$ in $\mathbb{R}$, consequently there is  $M>0$ such that for all $ j \in \mathbb{N}$
\begin{equation}\label{sh}
\Big|[\mathcal{A}(\mu_{j}(x,y))- \mathcal{A}(\mu(x,y))](\mu_{j}(x,y)-\mu(x,y))K(x,y)\Big|\leqslant M. 
\end{equation}
Thus, denoting 
$$V_{\mathcal{A}}:= [\mathcal{A}(\mu_{j}(x,y))\mu_{j}(x,y)+ \mathcal{A}(\mu(x,y))\mu(x,y)]K(x,y) \mbox{for all } j \in \mathbb{R},$$
 we get from \eqref{sh} that
\begin{equation*}
\begin{split}
V_{\mathcal{A}} \leqslant & M +  \mathcal{A}(\mu(x,y))\mu_{j}(x,y)K(x,y) +\mathcal{A}(\mu_{j}(x,y))\mu(x,y)K(x,y) \mbox{for all } j \in \mathbb{R}.
\end{split}
\end{equation*}
So using  $({a}_{2})$ and  $(\mathcal{K})$ in  inequality above, we have  $\mbox{for all } j \in \mathbb{R}$ that,
\begin{equation}\label{si}
\begin{split}
c_{\mathcal{A}}b_0\frac{|\mu_{j}(x,y)|^{p(x,y)}}{|x-y|^{N+sp(x,y)}}+c_{\mathcal{A}}b_0\frac{|\mu(x,y)|^{p(x,y)}}{|x-y|^{N+sp(x,y)}}\leqslant & M + C_{\mathcal{A}}b_1\frac{|\mu(x,y)|^{p(x,y)-1}|\mu_{j}(x,y)|}{|x-y|^{N+sp(x,y)}} \\ &+ C_{\mathcal{A}}b_1\frac{|\mu_{j}(x,y)|^{p(x,y)-1}|\mu(x,y)|}{|x-y|^{N+sp(x,y)}}.
\end{split}
\end{equation}
\noindent  Dividing \eqref{si} by $|\mu_{j}(x,y)|^{p(x,y)}$, we achieve
 \begin{equation}\label{estre}
 \begin{split}
\frac{c_{\mathcal{A}}b_0}{|x-y|^{N+sp(x,y)}} +\frac{c_{\mathcal{A}}b_0|\mu(x,y)|^{p(x,y)}}{|\mu_{j}(x,y)|^{p(x,y)}|x-y|^{N+sp(x,y)}}\leqslant & \frac{M}{|\mu_{j}(x,y)|^{p(x,y)}} \\& + \frac{C_{\mathcal{A}}b_1|\mu(x,y)|^{p(x,y)-1}}{|\mu_{j}(x,y)|^{p(x,y)-1}|x-y|^{N+sp(x,y)}} \\ &+ \frac{C_{\mathcal{A}}b_1|\mu(x,y)|}{|\mu_{j}(x,y)||x-y|^{N+sp(x,y)}}
\end{split}
\end{equation}
$\mbox{for all } j \in \mathbb{R}.$
Since  we are supposing that the sequence $(\mu_{j}(x,y))_{j \in \mathbb{N}}$  is unbounded, we can assume that $|\mu_{j}(x,y)| \to + \infty$ as $j \to + \infty$, then by  \eqref{estre} we obtain $c_{\mathcal{A}}b_0 \leqslant 0$ which is an contradiction.

\noindent Therefore,  the sequence $(\mu_{j}(x,y))_{j \in\mathbb{N}}$   is bounded in $\mathbb{R}$ and up to a subsequence $(\mu_{j}(x,y))_{j \in\mathbb{N}}$     converges to some $\nu \in \mathbb{R}$. Thus we obtain $\mu_{j}(x,y)\to \nu$ $\mbox{as } j \to + \infty$. Thence denoting $$\mathcal{U}(x,y):= [\mathcal{A}(\nu)- \mathcal{A}(\mu(x,y))](\nu -(\mu(x,y))K(x,y)$$ and using  $(a_{1})$  we conclude that
\begin{equation}\label{zeroa}
\mathcal{U}_{k_j}(x,y)\to \mathcal{U}(x,y) \mbox{ as } j \to + \infty.
\end{equation}
Consequently,  by \eqref{sg} and \eqref{zeroa}, we get
\begin{equation}\label{mono}
\mathcal{U}(x,y)= [\mathcal{A}(\nu)- \mathcal{A}(\mu(x,y))](\nu -(\mu(x,y))K(x,y)= 0.
\end{equation}
In this way, by strictly convexity  of $\mathscr{A}$, \eqref{mono} this occurs only if  $\nu =\mu(x,y)= u(x)-u(y)$. Therefore, by uniqueness of limit 
\begin{equation}\label{converi}
 u_{k_{j}}(x)-u_{k_{j}}(y)=\mu_{j}(x,y) \to\mu(x,y)= u(x)-u(y) \mbox{  in } \mathbb{R} \mbox{ as } j\to +\infty
\end{equation}
for almost all $ (x, y) \in \mathbb{R}^{N}\times \mathbb{R}^{N}$.

\noindent Now we consider   the sequence $(g_{k_j})_{j \in\mathbb{N}}$ in $L^{1}(\mathbb{R}^{N}\times \mathbb{R}^{N})$ defined pointwise for all $j \in \mathbb{N}$ by
\begin{equation*}
g_{k_j}(x,y):= \bigg[\frac{1}{2}\bigg(\mathscr{A}(\mu_{j}(x,y))+ \mathscr{A}(\mu(x,y))\bigg)- \mathscr{A}\bigg(\frac{ \mu_{j}(x,y)-\mu(x,y)}{2}\bigg)\bigg]K(x,y).
\end{equation*}
By convexity the map $ \mathscr{A}$ (see $(a_{1})$),  $g_{k_j}(x,y)\geqslant 0$ for almost all $(x,y)\in \mathbb{R}^{N}\times\mathbb{R}^{N}$. Furthermore, by continuity of map  $ \mathscr{A} $ (see $(a_{1})$) and \eqref{converi}, we have $$ g_{k_j}(x,y)\to \mathscr{A} (\mu(x,y))K(x,y) \mbox{ as } j \to +\infty \hspace{0.2cm}\mbox{ for all } \hspace{0.1cm}(x,y)\in \mathbb{R}^{N}\times\mathbb{R}^{N}.$$ 
 Therefore, using this above  information,  \eqref{ss4}, and Fatou's Lemma,  we get
\begin{equation*}
\begin{split}
\Phi(u)\leqslant \liminf_{j \to +\infty} g_{k_j}(x,y)= \Phi(u)- \limsup_{j \to +\infty}\int_{\mathbb{R}^{N}\times \mathbb{R}^{N}}\mathscr{A}\bigg(\frac{ \mu_{j}(x,y)-\mu(x,y)}{2}\bigg)K(x,y)\,dx\,dy.
\end{split}
\end{equation*}
Then 
\begin{equation}\label{lm1}
\limsup_{j \to +\infty}\int_{\mathbb{R}^{N}\times \mathbb{R}^{N}}\mathscr{A}\bigg(\frac{ \mu_{j}(x,y)-\mu(x,y)}{2}\bigg)K(x,y)\,dx\,dy\leqslant 0. 
\end{equation}
On the other hand, by $(a_{2})$, $(a_{3})$,  $(\mathcal{K})$, and Proposition \ref{lw0}, we infer that
\begin{equation}\label{sn}
\begin{split}
\int_{\mathbb{R}^{N}\times \mathbb{R}^{N}}\mathscr{A}\bigg(\frac{ \mu_{j}(x,y)-\mu(x,y)}{2}\bigg)K(x,y)\,dx\,dy \geqslant &  \frac{c_{\mathcal{A}}b_0}{2^{p^{-}}p^{+}} \inf \big\{ \| u_{k_j} - u \|_{\mathscr{W}}^{p^{-}}, \| u_{k_j} - u \|_{\mathscr{W}}^{p^{+}} \big\}\\\geqslant & 0  \mbox{  for all } j \in \mathbb{N}.
\end{split}
\end{equation}
Consequently, by \eqref{lm1} and \eqref{sn}, we achieve
\begin{equation*} 
\lim_{j\to+\infty}\| u_{k_j} - u \|_{\mathscr{W}} = 0. 
\end{equation*}
Therefore,   we can conclude that $u_{k_j} \to u$ in $\mathscr{W}$ as $j \to +\infty$. Since $(u_{k_j})_{j\in \mathbb{N}}$ is an arbitrary subsequence of $(u_k)_{k \in \mathbb{N}}$, this shows that  $u_k \to u$  as $k \to +\infty$ in  $\mathscr{W}$, as required.
  \end{proof}
 \section*{Acknowledgements}
  \hfill \break
 The first author was financed  by the  Coordena\c{c}\~ao de Aperfei\c{c}oamento de Pessoal de N\'ivel Superior(CAPES)-Finance Code 001. The second author was supported the  Coordena\c{c}\~ao de Aperfei\c{c}oamento de Pessoal de N\'ivel Superior(CAPES)-Finance Code 001. The third author named was partially supported by INCTMAT/CNPq/Brazil and CNPq/Brazil 307061/2018-3.

\end{document}